%% file: EnumerationPaper.tex
\documentclass[11pt]{amsart}

\usepackage{amsmath, amssymb, amsthm, amsfonts, graphicx, pinlabel}

\input{Sections/Preamble}

\begin{document}

\title[Ruling polynomials and augmentations over finite fields]{Ruling polynomials and augmentations over finite fields}

\author{Michael B. Henry}
\address{Siena College, Loudonville, NY 12211}
\email{mbhenry@siena.edu}

\author{Dan Rutherford}
\address{University of Arkansas, Fayetteville, AR 72701}
\email{drruther@uark.edu}

\begin{abstract}
For any Legendrian link, $L$, in $(\R^3, \ker(dz-y\,dx))$ we define invariants, $\mathit{Aug}_m(L,q)$, as normalized counts of augmentations from the Legendrian contact homology DGA of $L$ into a finite field of order $q$ where the parameter $m$ is a divisor of twice the rotation number of $L$.  Generalizing a result from \cite{Ng2006} for the case $q =2$, we show the augmentation numbers, $\mathit{Aug}_m(L,q)$, are determined by specializing the $m$-graded ruling polynomial, $R^m_L(z)$, at $z = q^{1/2}-q^{-1/2}$.   As a corollary, we deduce that the ruling polynomials are determined by the Legendrian contact homology DGA.  
\end{abstract}

\maketitle

\input{Sections/Introduction}
\input{Sections/Background2}
\input{Sections/Outline2}
\input{Sections/MCSsSRFormsDiskEquations2}

\input{Sections/AFormsAndAugmentations}
\input{Sections/SRFormAFormBijection}

\addcontentsline{toc}{section}{Bibliography} 
\bibliographystyle{amsplain}
\bibliography{Bibliography}

\end{document}

%% file: Sections/Introduction.tex
\section{Introduction}
\mylabel{s:intro}

The main result of this article shows that, for  a Legendrian link $L$ in standard contact $\R^3$,   normalized counts of augmentations of the Legendrian contact homology algebra with values in finite fields can be recovered from specializing the ruling polynomials of $L$. To provide some context, we begin with a brief discussion of these two invariants of Legendrian links.

  The Legendrian contact homology algebra, $(\A,\partial)$, is an invariant of Legendrian submanifolds that is defined in a number of settings including in higher dimensions. The invariant is a differential graded algebra (DGA) whose differential counts certain holomorphic disks in the symplectization of the ambient contact manifold with boundary on the Lagrangian cylinder over the Legendrian submanifold.  In the context of Legendrian links in $\R^3$, the DGA $(\A, \partial)$ is also known as the Chekanov-Eliashberg algebra as Chekanov gave an equivalent (see \cite{Chekanov2002a, Eliashberg, Etnyre2002}) formulation of the DGA from a purely combinatorial perspective, and showed that it could be used to distinguish Legendrian knots that were previously not known to be distinct. 

Aside from serving as an effective invariant of Legendrian links, the Chekanov-Eliashberg algebra is also a natural tool for studying Legendrian knots in the context of symplectic topology.  The algebra has been shown to provide information about standard symplectic objects associated to $L$ such as Lagrangian fillings \cite{Chantraine,Ekholm2012,EkholmHK} and generating families \cite{Fuchs2008,Henry2013}.  In addition,  the Chekanov-Eliashberg algebra is also a fundamental ingredient for the combinatorial computation of symplectic homology of Weinstein $4$-manifolds using results from \cite{Bourgeois2012}. 

A normal ruling is a type of decomposition of the front diagram of a Legendrian knot that was originally defined independently, and with different terminologies, in \cite{Chekanov2005} and \cite{Fuchs2003}.  Chekanov and Pushkar showed that refined counts of normal rulings provide Legendrian knot invariants, and it is convenient to present their invariants by collecting them as the coefficients of Laurent polynomials in $z$, known as the $m$-graded ruling polynomials, $R^m_L(z)$.  Here, $m$ is a common divisor of twice the rotation number of every component of $L$.    Ruling polynomials are fairly natural from the perspective of knot theory as they are characterized by skein relations.  However, their invariance has been somewhat mysterious from the point of view of contact geometry, and generalizations of these invariants to higher dimensions have so far not been forthcoming.  

Connections between normal rulings and the Chekanov-Eliashberg algebra began to appear already in \cite{Fuchs2003} with further developments in  \cite{Fuchs2004, Ng2013, Ng2006, Sabloff2005}.  In particular, it is shown in \cite{Fuchs2003, Fuchs2004, Sabloff2005} that the existence of an $m$-graded normal ruling is equivalent to the existence of an $m$-graded augmentation of $(\A,\partial)$ into $\Z/2$. In general, an augmentation is an algebra homomorphism from $(\A,\partial)$ to $(R, 0)$ that is supported in degrees congruent to $0$ mod $m$ and, in addition, commutes with the differential.  Here, $R$ denotes a ground ring 
equipped with $0$ as a differential.  Results of the present article, see Theorem \ref{thm:structure}, show that the equivalence of the existence of normal rulings and augmentations holds when considering augmentations into any field.  Note that recent work of Leverson, \cite{Leverson2014}, independently establishes this result as well as equivalence in the case of augmentations into $\Z$.

In the case when $m=0$ or $m$ is odd, Ng and Sabloff \cite{Ng2006} defined numerical invariants, which we denote here by $\widetilde{\mathit{Aug}}_m(L,\Z/2\Z)$, that are normalized counts of augmentations into $\Z/2\Z$.  In addition, they showed that these augmentation numbers are determined by the ruling polynomials via 
\[
\widetilde{\mathit{Aug}}_m(L,\Z/2\Z) = R^m_L(2^{-1/2}).
\] 
The present article provides a significant extension of this result.

\subsection{Statement of results}  Let $L = \cup_{i=1}^cL_i$ be a Legendrian link with rotation number $r(L)=  \gcd\{r(L_i) \}$.  For any divisor $m$ of $2 r(L)$ (the restriction that $m=0$ or $m$ is odd is now removed) and any finite field, $\F_q$, of order $q$ we associate an $\F_q$-valued $m$-graded augmentation number to $L$ denoted 
\[
\mathit{Aug}_m(L,q).
\]
The augmentation numbers are Legendrian link invariants,  see Theorem \ref{thm:augnumbers}, and they have the following relation with the ruling polynomials of $L$.

\begin{theorem} \label{thm:Main} Let $L \subset \R^3$ be a Legendrian link with a chosen Maslov potential.  For any $m \,|\, 2 r(L)$ and any prime power $q$, we have   
\begin{equation} \label{eq:11}
\mathit{Aug}_m(L, q) = q^{-[d+c]/2}z^{c}R^m_L(z)
\end{equation}
where $z = q^{1/2}-q^{-1/2}$ with $q^{1/2}$ the positive square root; $c$ denotes the number of components of $L$; and  $d$ denotes the maximum degree in $z$ of $R^m_L(z)$.
\end{theorem}

As a corollary, we can see that the Chekanov-Eliashberg algebra determines the ruling polynomials, a result that has been anticipated in the literature at least going back to \cite{Ng2003}.  For this purpose, view the augmentation numbers collectively as a function from the set of prime powers to $\R$.

\begin{corollary}  The augmentation function $\mathit{Aug}_m(L,\cdot)$ and the $m$-graded ruling polynomial $R^m_L(z)$ are equivalent invariants of Legendrian links with a given number of components. 
\end{corollary}

\begin{proof}
Suppose $K$ and $L$ are Legendrian links with $c(K)$ and $c(L)$ components, respectively, and suppose $c(K) = c(L)$ holds.  Theorem \ref{thm:Main} shows that if $R^m_K(z) = R^m_L(z)$, then the augmentation functions must agree as well.  For the converse, note that $z^{c(L)}R^m_L(z)$ is a polynomial in $z$ (see Remark \ref{rem:RP} below), so $f_L(q^{1/2}) := q^{-[d(L)+c(L)]/2}z^{c(L)}R^m_L(z)$ is a Laurent polynomial in $q^{1/2}$, where $z = q^{1/2}-q^{-1/2}$.  Hence, if $K$ and $L$ have identical augmentation functions, then $f_K(q^{1/2})$ and $f_L(q^{1/2})$ agree at infinitely many points and are therefore equal.  Note that $d(L) = \deg_z( z^{c(L)}R^m_L(z)) - c(L) = \mbox{max-deg}_q f_L(q^{1/2}) - \mbox{min-deg}_q f_L(q^{1/2}) -c(L)$, so it follows that $d(K) = d(L)$.  We can therefore conclude that $R^m_K(z) = R^m_L(z)$.
\end{proof}

Theorem \ref{thm:Main} gives an alternate explanation for the Legendrian isotopy invariance of the ruling polynomials via the invariance of the augmentation function.  The definition of $\mathit{Aug}_m(L,\cdot)$ extends in a straightforward manner to Legendrians in higher dimensional $1$-jet spaces such as $\R^{2n+1}$, and it is interesting to ask if this function may be determined by a polynomial in $z = q^{1/2}-q^{-1/2}$ in general.

\begin{remark}
The cases $m = 1$ and $m=2$ are of some special interest due to a connection between the ruling polynomials and the Kauffman and HOMFLY-PT polynomials, which are invariants of smooth knots that are Laurent polynomials in $a$ and $z$.  Using conventions from \cite{Rutherford2006}, $R^2_L(z)$ (resp. $R^1_L(z)$) is equal to the coefficient of $a^{-tb(L)-1}$ in the HOMFLY-PT polynomial (resp. Kauffman polynomial) where $tb(L)$ denotes the Thurston-Bennequin number.  Note that the same substitution $z = q^{1/2}- q^{-1/2}$ is also commonly applied to these knot polynomials, for instance, in recovering the quantum $\mathfrak{sl}_n$ invariants from the HOMFLY-PT polynomial.    
\end{remark}

\subsection{Outline of the article} 
In Section 2 we recall some background in connection with normal rulings and the Chekanov-Eliashberg algebra, and in particular give our sign conventions for defining the algebra over $\Z$.   Section 3 contains a definition of augmentation numbers where we use the dimension of an augmentation variety over $\C$ to provide a normalizing factor.  In the remainder of the section a proof of Theorem \ref{thm:Main} is presented assuming a result, Theorem \ref{thm:structure}, which may be of some independent interest as it concerns the structure of the augmentation varieties over an arbitrary field.  

The remainder of the paper focuses on establishing Theorem \ref{thm:structure}.  Throughout, we make use of Morse complex sequences (abbr. MCSs) which were defined over $\Z/2\Z$ in \cite{Henry2011,Pushkar'a}.  Our approach is strongly motivated by two standard forms for Morse complex sequences, the SR-form and the A-form, that were considered over $\Z/2\Z$ in \cite{Henry2011} and are related to normal rulings and augmentations respectively.  Using SR-form MCSs we are able to realize the part of the augmentation variety that corresponds to a particular normal ruling as the solution set to a system of equations which is easier to analyze than the system of equations characterizing the full variety.  The paper concludes with the construction of bijections between A-form MCSs and augmentations in Section 5 and between SR-form and A-form MCSs in Section 6.
 
\subsection{Acknowledgments}
Our collaboration has been stimulated by our participation in a SQuaRE research group at the American Institute of Mathematics.  We thank AIM and also the other SQuaRE members Dmitry Fuchs, Paul Melvin, Josh Sabloff, and Lisa Traynor.  We also thank Mark Johnson, Yo'av Rieck, Michael Sullivan, and Jeremy Van Horn-Morris for useful conversations.

%% file: Sections/Background2.tex
\section{Background}
\mylabel{s:Background}

We will use a few basic concepts from algebraic geometry as follows.  Let $\F$ be a field, and denote $\F^\times = \F \setminus \{0\}$.  Using coordinates $(x_1, \ldots, x_m, y_1, \ldots, y_n)$ on $(\F^\times)^m \times \F^n$, an \textbf{affine algebraic set} or \textbf{affine variety} in $(\F^\times)^m \times \F^n$ is the common zero locus of some collection of Laurent polynomials in $\F[x_1, x_1^{-1}, \ldots, x_m, x_m^{-1}, y_1, \ldots, y_n]$; negative powers are allowed for the $x_i$ but not for the $y_i$.  A {\bf regular map} from an affine variety $V \subset (\F^\times)^{m_1} \times \F^{n_1}$ to an affine variety $W\subset (\F^\times)^{m_2} \times \F^{n_2}$ is a function from $V$ to $W$ that is the restriction of a polynomial map $f :(\F^\times)^{m_1} \times \F^{n_1} \rightarrow \F^{m_2} \times \F^{n_2}$ that may include negative powers of the $x_i$ coordinates.  Affine varieties $V$ and $W$ are isomorphic, written $V \cong W$, if there is a regular bijection between them whose inverse is also regular.
Finally, we need that an affine variety over $\C$ has a {\bf dimension}, and in the case that $V$ is non-singular and irreducible this agrees with the dimension of $V$ as a complex manifold.  

In addition, we assume some familiarity with Legendrian knots in $\R^3$ with its standard contact structure $\xi= \text{ker}(dz-ydx)$.  In particular, the reader should be familiar with front ($xz$) and Lagrangian ($xy$) projections as well as with the rotation number, cf. \cite{Etnyre2005}.  For a multi-component Legendrian link $L= \cup_{i=1}^c L_i$ we use the convention that $r(L) = \gcd \{r(L_i) \,|\, 1 \leq i \leq c\}$.  
We say that a front diagram is \textbf{plat} if every left cusp has the same $x$-coordinate, every right cusp has the same $x$-coordinate, and all self-intersections are transverse double-points. A \textbf{nearly plat} front diagram is the result of perturbing a plat front diagram slightly so that each cusp and crossing has a distinct $x$-coordinate.  An arbitrary Legendrian link is Legendrian isotopic to a link with nearly plat front diagram, and we will often make this assumption on front diagrams to simplify proofs. In a small neighborhood $U \subset \R^2$ of a cusp, the \textbf{upper} (resp. \textbf{lower}) strand of the cusp is the connected component of $U \cap ( L \setminus \{ \mbox{cusp points} \})$ with larger (resp. smaller) $z$-coordinate. A \textbf{Maslov potential} for $L$ is a locally constant map $\Maslov : L \setminus \{ \mbox{cusp points} \} \to \Z/ 2r(L) \Z$ satisfying 
\[
\Maslov( \mbox{upper strand} ) = \Maslov( \mbox{lower strand} ) + 1 \mod{2 r(L)}
\]
 near a cusp point of the front diagram.

\subsection{$m$-Graded Normal Rulings}
\mylabel{ss:NormalRulings}

Suppose $\front$ is the front diagram of a Legendrian link, and that all crossings and cusps of $\front$ have distinct $x$-coordinates. Label the $x$-coordinates of the crossings and cusps $x_0 < x_1 < \hdots < x_n$.  For each $1 \leq i \leq n$, label the strands of the tangle $D \cap \{ (x,z) \in \R^2 : x_{i-1} < x < x_{i} \}$ from \emph{top to bottom} $1, 2, \hdots, s_i$.

\begin{definition} \label{def:NR}
 A \emph{normal ruling} of $D$ is an $n$-tuple $\ruling = ( \ruling_1, \ruling_2, \hdots, \ruling_n)$ satisfying:
\begin{enumerate}
	\item  For each $1 \leq i \leq n$, $\ruling_i$ is a fixed-point free involution of $\{1, 2, \hdots, s_i\}$.
	\item If $x_i$ is the $x$-coordinate of a left cusp between strands $k$ and $k+1$, then $\ruling_{i+1}(k) = k+1$ and the restriction of $\ruling_{i+1}$ to the remaining strands agrees with $\ruling_{i}$ when we identify \\ $\{1, \ldots, s_{i+1}\}\setminus\{k,k+1\}$ with $\{1, \ldots, s_{i}\}$ in the obvious way. 

An analogous requirement is imposed at  right cusps.
	\item Suppose $x_i$ is the $x$-coordinate of a crossing between strands $k$ and $k+1$. Then $\ruling_i(k) \neq k+1$ and either:
	
\begin{enumerate}
	\item $$ \ruling_{i+1} = (k \hspace{2mm} k+1) \circ \ruling_{i} \circ (k \hspace{2mm} k+1)$$ where $(k \hspace{2mm} k+1)$ denotes the transposition; or  
	\item $$\ruling_i = \ruling_{i+1}.$$ In this case, we also require $\ruling_i$ satisfies one of the following:
\begin{enumerate}
	\item $\ruling_i(k) < k < k+1 < \ruling_i(k+1)$;
	\item $k < k+1 < \ruling_i(k+1) < \ruling_i(k)$; or 
	\item $\ruling_i(k+1) < \ruling_i(k) < k < k+1$.
\end{enumerate}
		
\end{enumerate}
	 
\end{enumerate}

A crossing with $x$-coordinate $x_i$ is called a \textbf{switch} if $\ruling$ satisfies (3b) at $x_i$, and the requirement at switches specified by (3b) (i)-(iii) is known as 
the \textbf{normality condition}.   A crossing that is not a switch is called a \textbf{departure} (resp. a \textbf{return}) if the normality condition is satisfied before (resp. after) the crossing.
\end{definition}

Suppose $L$ is equipped with a Maslov potential, $\mu$.  Then, each crossing $q$ of $D$ is assigned the {\bf degree}, $|q| \in \Z/2r(L)\Z$, defined as the difference of the Maslov potential on the overstrand and understrand of the crossing.

\begin{definition} \label{def:MP}
Let $m \geq 0$ be a divisor of $2 r(L)$.  A normal ruling $\rho$ of $L$ is \textbf{$m$-graded} if all switches of $\rho$ have degree congruent to $0$ mod $m$.  Returns (resp. departures) of $\rho$ with degree $0$ mod $m$ are referred to as \textbf{$m$-graded returns} (resp. \textbf{$m$-graded departures}).
\end{definition}

\begin{figure}[t]
\centering
\includegraphics[scale=.65]{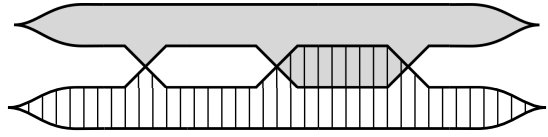}
\caption{A presentation of a normal ruling in terms of disks.}
\label{fig:trefoilruling}
\end{figure}

\begin{figure}[t]
\labellist
\small\hair 2pt
\pinlabel {(S1)} [tl] at 68 110
\pinlabel {(S2)} [tl] at 259 110
\pinlabel {(S3)} [tl] at 453 110
\pinlabel {(R1)} [tl] at 68 0
\pinlabel {(R2)} [tl] at 259 0
\pinlabel {(R3)} [tl] at 453 0
\endlabellist
\centering
\includegraphics[scale=.65]{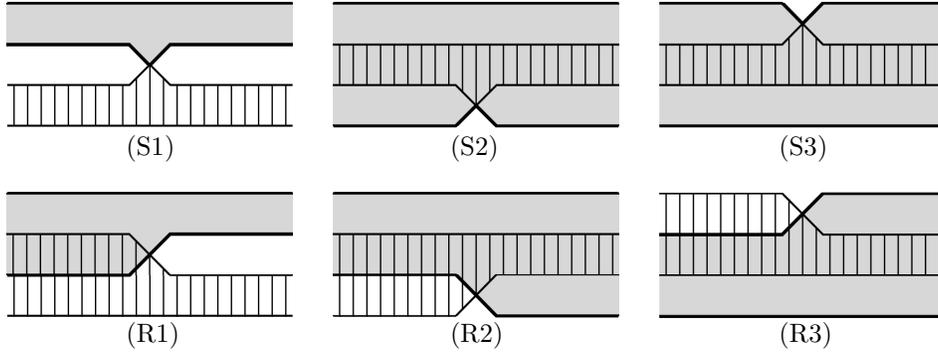}
\caption{The behavior of ruling disks at a switch (top row) or a return (bottom row). The three types of departures can be seen by reflecting each of (R1)-(R3) across a vertical axis.}
\label{fig:returns}
\end{figure}

\subsubsection{Normal rulings as decompositions of $\front$.} \label{sec:disks} Since the involutions $\rho_i$ are fixed point free, they divide the strands of $\front$ into pairs in vertical strips of the $xz$-plane that do not contain crossings or cusps.  Thus in each such strip, the ruling assigns to each strand a corresponding {\bf companion strand} of the ruling.  The definition allows us to extend this local pairing by covering the entire front diagram with pairs of continuous paths with monotonic $x$-coordinates.  Each pair of paths meet only at shared cusp endpoints; paths that belong to distinct pairs can meet only at crossings where they either cross each other transversely (as in (3a)) or both turn a corner if the crossing is a switch.  Topologically, each pair of paths bounds a disk in the $xz$-plane, and in figures it is convenient to present normal rulings by shading these disks; see Figure \ref{fig:trefoilruling} for an example.  

The normality condition allows us to divide switches into three distinct types, (S1)-(S3) as indicated in Figure \ref{fig:returns}.  For values of $x$ near a switch the two disks that meet at the switch are either disjoint as in Type (S1) or one is nested inside the other as in Types (S2) and (S3).   At a departure (resp. return) the configuration of the corresponding disks changes from nested or disjoint to interlaced (resp. interlaced to nested or disjoint).  The returns can as well be divided into three types, (R1)-(R3), also indicated in Figure \ref{fig:returns}.

\begin{definition}
The set of $m$-graded normal rulings of a Legendrian link $\Leg$ with Maslov potential $\Maslov$ is denoted $\mathcal{R}^m(\Leg, \Maslov)$. The \textbf{$m$-graded ruling polynomial} of $\Leg$ with Maslov potential $\Maslov$ is $$R_{\Leg, \Maslov}^m (z) = \sum_{\ruling \in \mathcal{R}^m(\Leg, \Maslov)} z^{j(\ruling)},$$ where $j(\ruling):=\#(\mbox{switches}) - \#(\mbox{right cusps})$.
\end{definition}

It is shown in \cite{Chekanov2005} that the $m$-graded ruling polynomials are Legendrian isotopy invariants.  Note that for multi-component links they may depend on the choice of Maslov potential $\mu$.  The grading of the Chekanov-Eliashberg algebra depends as well on a choice of Maslov potential, and in Theorem \ref{thm:Main} it should be understood that the same $\mu$ is used in both contexts.

\begin{remark}  \label{rem:RP}
Although $R_{\Leg,\Maslov}^m (z)$ may contain negative powers of $z$, the product $z^c \cdot  R_{\Leg,\Maslov}^m (z)$ does not where $c$ is the number of components of $L$.  To see this, for a given normal ruling $\rho$, resolve all switches to parallel, horizontal lines in the front projection.  The result is a link of Legendrian unknots with one component for each disk of the ruling.  The addition of each switch can only decrease the number of components by $1$, and this observation gives the inequality $c \geq -j(\rho)$.   
\end{remark}

\subsection{The Chekanov-Eliashberg algebra over $\Z[H_1(L)]$}
\mylabel{ss:CE-DGA}

We now recall a definition of the Chekanov-Eliashberg algebra with coefficients in the group ring $\Z[H_1(L)]$.  The sign conventions used here follow  \cite{Ng2010}.  

Let $L$ be an oriented Legendrian link with $c$ components.  Choosing an ordering of the components, $L = L_1 \sqcup \cdots \sqcup L_c$, provides an isomorphism $\Z[H_1(L)] \cong \Z[t_1, t_1^{-1}, \ldots, t_c, t_c^{-1}]$.  In addition, choose a marked point on each component, $*_i \in L_i$ for $1 \leq i \leq c$.     

It is most natural to use the Lagrangian projection (projection to the $xy$-plane) of $L$ when defining the Chekanov-Eliashberg algebra.  In this article, we will make use of Ng's  resolution construction  that produces a Lagrangian projection of a link Legendrian isotopic to $L$.  This Lagrangian projection is obtained from the front diagram $D$ by placing the strand with lesser slope on top at crossings; smoothing the cusps of $D$; and then adding an extra negative half twist near each right cusp:
\[
\includegraphics[scale=.4]{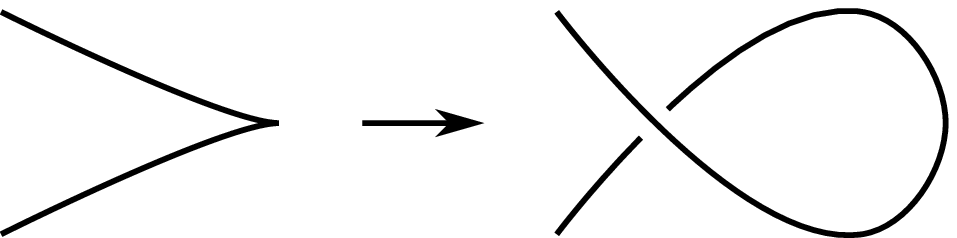}.
\]
See \cite{Ng2003} for details of the resolution construction.
We will use the notation $(\A(D), \partial)$ for the Chekanov-Eliashberg DGA associated to the resolution of the front diagram $D$.  This is a $\Z/2r(L)\Z$-graded algebra with differential $\partial$.

As an algebra, $\A(D)$ is the free associative (non-commutative) $\Z[t_1, t_1^{-1}, \ldots, t_c, t_c^{-1}]$ unital algebra generated by the crossings of the Lagrangian projection associated with $D$. These crossings are in correspondence with the crossings of $D$ itself and the right cusps of $D$, and we label them as $q_1, \ldots, q_N$.  At each crossing we associate Reeb signs to each of the quadrants of the crossing as in Figure \ref{fig:Reeb}.  In addition, each quadrant is given an orientation sign, so that the two quadrants which sit to the right of the understrand (with respect to its orientation) both have negative orientation signs and the other two quadrants have positive orientation signs; see Figure \ref{fig:Reeb}.

\begin{figure}[t]
\labellist
\small\hair 2pt
\pinlabel {$-$} [b] at 40 34
\pinlabel {$-$} [t] at 40 16
\pinlabel {$+$} [l] at 54 25
\pinlabel {$+$} [r] at 26 25
\endlabellist
\centering
\includegraphics[scale=.65]{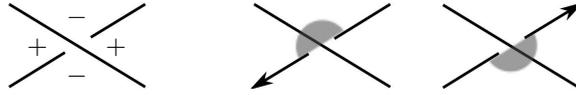}
\caption{Reeb signs (left) and orientation signs (right) at crossings of a Lagrangian projection.  The shaded quadrants indicate negative orientation signs.}
\label{fig:Reeb}
\end{figure}

A $\Z/2r(L)\Z$-grading may be defined on $\A(D)$ as follows. Choose a Maslov potential for $L$.  Then, a generator $q$ corresponding to a crossing of $D$ is assigned the same degree $|q| \in \Z/2r(L)\Z$ as in the discussion above Definition \ref{def:MP}.  All generators corresponding to right cusps have degree $1$, and the group ring generators have $|t_i| =0$.  The grading is extended to $\A(D)$ so that $|a\cdot b| = |a| + |b|$ when $a$ and $b$ are homogeneous.  We remark that for a $1$-component link the grading is independent of the choice of Maslov potential, but in the multi-component case different choices of $\mu$ can produce different gradings on $\A(D)$. 

Next, we define the differential, $\partial$.  For each $\ell \geq 0$, fix $\ell+1$ points, $z_0, z_1, \ldots, z_\ell$ appearing in counterclockwise order along the boundary of the unit disk $D^2 \subset \R^2$ and let $D^2_\ell = D^2 \setminus \{z_0, z_1,\ldots, z_\ell \}$.  Given generators $a, b_1, \ldots, b_n$ of $\A(D)$ we consider orientation preserving immersions of $D^2_\ell$ into the plane such that the boundary of $D^2_\ell$ maps to the Lagrangian projection of $L$ and a suitably small neighborhood of $z_0$ (resp. $z_i$ with $1 \leq i \leq \ell$) maps to a single  quadrant of $a$ (resp. $b_i$) with positive (resp. negative) Reeb sign.  
We let $\Delta(a, b_1, \ldots, b_n)$ denote the set of all such immersions modulo orientation preserving reparametrization.  To each disk $f \in \Delta(a, b_1, \ldots, b_n)$ we associate the word
\begin{equation} \label{eq:word}
w(f) = \epsilon(f) t_1^{n_1(f)}\cdots t_c^{n_c(f)} b_1 \cdots b_n 
\end{equation}
where $n_i(f)$ denotes the number of times the boundary of $D^2_\ell$ intersects the base point, $*_i$, counted with sign, i.e. the algebraic intersection number of $\partial f$ and $*_i$ with respect to the orientation of  $L_i$.
Moreover, $\epsilon(f)$ is a sign defined by 
\[
\epsilon(f) = \epsilon' \epsilon_0 \epsilon_1 \cdots \epsilon_n
\] 
where $\epsilon_i$ is the orientation sign of the $i$-th corner of $f$ and $\epsilon'$ is $+1$ if the orientation of the initial arc of $f$ (coming out of $a$) agrees with the orientation of $L$ and $-1$ otherwise.

The differential $\partial: \A(D) \rightarrow \A(D)$ is then defined on generators by
\[
\partial a = \sum_{n \geq 0} \sum_{b_1, \ldots, b_n} \sum_{f \in \Delta(a, b_1, \ldots, b_n)}  w(f)
\]
and extended to all of $\A(D)$ according to the Liebniz rule $\partial( xy) = (\partial x) y + (-1)^{|x|} x (\partial y)$.

\begin{theorem}[\cite{Chekanov2002a, Etnyre2002}] \label{thm:CEDGA}  The stable tame isomorphism type of the differential graded algebra, $(\A(D), \partial)$ is a Legendrian isotopy invariant of $L$.
\end{theorem}

This invariance statement requires some explanation. An {\bf algebra stabilization} of $(\A(D), \partial)$ is a DGA of the form $(S\A(D), \partial')$ where $S\A(D)$ denotes an algebra obtained from $\A(D)$ by adding two new generators $e$ and $f$ with $|e| = |f|+1$; the differential satisfies 
$\partial' e =f$ and $\partial'|_{\A(D)} = \partial$.  
Theorem \ref{thm:CEDGA} states that if $D_1$ and $D_2$ are front diagrams representing Legendrian isotopic knots, then the DGA's $(\A(D_1), \partial_1)$ and $(\A(D_2), \partial_2)$ will become isomorphic after possibly stabilizing each algebra some number of times.  
Moreover, the isomorphism may be assumed to be of a particular type known as a tame isomorphism, but this will not be important for our purposes. 
In the multi-component case, the grading depends on a choice of Maslov potential.  Once we choose a Maslov potential for $D_1$ there is a corresponding Maslov potential for $D_2$ with respect to which stable isomorphism holds.  

\begin{remark}
The Chekanov-Eliashberg algebra was originally defined over $\Z/ 2 \Z$ in \cite{Chekanov2002a}.  A generalization to $\Z[H_1(L)]$ coefficients was first established in \cite{Etnyre2002}.  Distinct (but equivalent) conventions for defining the Chekanov-Eliashberg over $\Z[H_1(L)]$ have appeared elsewhere in the literature with our conventions related to those of \cite{Ng2010} as follows.  The article \cite{Ng2010} restricts attention to one component links.  There, a more general invariant of Legendrian knots is constructed, the LSFT algebra, which is an algebra $\hat{\mathcal{A}}$ with filtration $\mathcal{F}^0\hat{\mathcal{A}} \supset  \mathcal{F}^1\hat{\mathcal{A}} \supset \cdots$ by ideals.  The quotient $\A = \mathcal{F}^0\hat{\mathcal{A}}/\mathcal{F}^1\hat{\mathcal{A}}$ is a $\Z$-graded algebra, and the construction results in a differential $\partial$ defined on $\A$.  After taking a further quotient to allow the group ring generator $t$ to commute with the other generators and  reducing the grading mod $2r(L)$, this is precisely the algebra $(\A(D), \partial)$ that we have defined above.  Appendix A of \cite{Ng2010} shows that $(\A(D), \partial)$ is equivalent to other versions of the Chekanov-Eliashberg algebra arising from alternate sign conventions.  Finally, note that \cite{Ng2003} contains a discussion of the multi-component case. 
\end{remark}

%% file: Sections/Outline2.tex
\section{Augmentation numbers}
\mylabel{s:Outline}

Using Theorem \ref{thm:CEDGA}, many effective Legendrian knot invariants can be extracted from the Chekanov-Eliashberg algebra.  In this section, we define augmentation numbers which arise from counting homomorphisms from $(\A(D), \partial)$ into a finite field, and we outline the proof of Theorem~\ref{thm:Main}.

\begin{definition}  
\mylabel{defn:aug}
Let $\F$ be a field and $D$ be a front diagram for a Legendrian link $L$ with chosen Maslov potential.  An {\bf $\F$-valued augmentation} is a ring homomorphism $\varepsilon : (\A(D), \partial) \rightarrow (\F, 0)$ satisfying $\varepsilon(1) = 1$ and $\varepsilon \circ \partial = 0$.  In addition, for a given divisor $m \, | \,2 r(L)$, we say $\varepsilon$ is {\bf $m$-graded} if we have $|x| = 0$ mod $m$ for any generators such that $\varepsilon(x) \neq 0$. 
Moreover, we use the notation $\overline{\mathit{Aug}}_m(D; \F)$ to denote the set of all $m$-graded, $\F$-valued augmentations.
\end{definition}

After ordering the generators $t_1, \ldots, t_c$ of the coefficient ring and the generators, $q_1, \ldots, q_N$, of the Chekanov-Eliashberg algebra 
there is an injective map
\[
\Omega: \overline{ \mathit{Aug}}_m(D, \F) \hookrightarrow (\F^\times)^c \times \F^N,  \quad \varepsilon \mapsto (\varepsilon(t_1), \ldots, \varepsilon(t_c), \varepsilon(q_1), \ldots, \varepsilon(q_N)).
\]
We denote the image of $\Omega$ by $V_m(D,\F)$ and refer to it as the {\bf augmentation variety} of $(\A(D), \partial)$ over $\F$.  Notice that $V_m(D,\F)$ is an affine algebraic set since it is the common zero locus of the collection of polynomials $\{\partial q_1, \ldots, \partial q_N\} \cup \{ q_i \, :\,  |q_i| \neq 0 \mod m \} \subset \F[t^{\pm1}_1, \ldots, t^{\pm1}_c, q_1, \ldots, q_N]$ .

When $\F_q$ is a finite field, of order $q$, we define the {\bf augmentation numbers} of $L$ by
\[
\mathit{Aug}_m(L,q) = q^{-dim_\cc V_m(D, \cc)} |V_m(D, \F_q)|
\]
where $|V_m(D, \F_q)|$ denotes the number of elements of $V_m(D, \F_q)$. 
\begin{theorem} \label{thm:augnumbers} For any $m \,|\, 2r(L)$ and any prime power $q$, the augmentation number $\mathit{Aug}_m(L,q)$ is a Legendrian isotopy invariant of $L$.
\end{theorem}
\begin{proof}
A tame isomorphism between DGA's induces an isomorphism of the corresponding augmentation varieties for any $\F$.  Thus, according to Theorem \ref{thm:CEDGA} it suffices to show that stabilizing the Chekanov-Eliashberg algebra from $(\A(D), \partial)$ to $(S\A(D), \partial')$ does not change the augmentation number.  Let $V_m(\A(D), \F)$ and $V_m(S\A(D),\F)$ denote augmentation varieties corresponding to these two algebras.  Suppose that the stabilization adds generators $e$ and $f$ in degrees $k$ and $k-1$ respectively.  Since $\partial' e = f$, the augmentation varieties are equal unless $k = 0 \mod m$ in which case the augmentation variety of the stabilized algebra satisfies $V_m(S\A(D),\F) = V_m(\A(D), \F) \times \F$.  Thus, if $k = 0 \mod m$, $\dim_\cc V_m(S\A(D),\cc) = \dim_\cc V_m(\A(D),\cc) +1$, and $|V_m(S\A(D),\F_q)| = q \cdot |V_m(\A(D),\F_q)|$, so the result follows.
\end{proof}

\begin{remark}
\begin{itemize}
\item[(i)]  It is possible to lift the $\Z/2r(L)\Z$-grading of $(\A(D),\partial)$ to a $\Z$-grading where the group ring generators are assigned the degree $|t_i| = 2r(L_i)$.  However, for our purposes there does not appear to be any advantage in doing so.  An augmentation $\varepsilon$ must take the $t_i$ to invertible elements of $\F$, and therefore can still only be $m$-graded when $m|2r(L)$.  
\item[(ii)] The next two items briefly recall from the literature two other options for defining Legendrian isotopy invariant augmentation numbers.  Working over $\Z/2\Z$, Ng and Sabloff in \cite{Ng2006} use a product $2^{-\chi_m^*(D)/2} |\overline{Aug}_m(D, \Z/2\Z)|$ to define augmentation numbers in the case  $m$ is odd or $m=0$.  Here, 
\[
\chi^*_m(D) = \sum_{k = 0}^{m-1} (-1)^k a_k \mbox{  if $m$ is odd, and  }  \chi^*_0(D) = \sum_{k \geq 0} (-1)^k a_k + \sum_{k < 0} (-1)^{k+1} a_k   
\]
where $a_k$ denotes the number of generators of $\A(D)$ with degree congruent to  $k$ mod $m$ including the generators $t_i$ of the coefficient ring.  Their definition extends to define augmentation numbers for any finite field by
\[
\widetilde{Aug}(L,q) = q^{-\chi_m^*(D)/2} |\overline{Aug}_m(D, \F_q)|.
\]
Invariance follows because $\chi_m^*(D)$ increases by $2$ when the algebra is stabilized with new generators in degree $m$ and $m-1$ and is unchanged by all other stabilizations.

The methods of this paper, with Lemma 5 from \cite{Ng2006} used in place of our Lemma \ref{lem:returns2}, can establish  
\[
\widetilde{Aug}(L,q) = (q^{1/2} - q^{-1/2})^{c} R^m_L(q^{1/2}- q^{-1/2})
\]
with $c$ the number of components of $L$.  Note that it is unclear how to extend Ng and Sabloff's definition to the case when $m>0$ is even.

\item[(iii)]  There is a notion of homotopy for augmentations, and this provides an equivalence relation on the set of augmentations for a given DGA.  In \cite{Henry2011}  homotopy classes of augmentations for a Legendrian knot are studied in connection with Morse complex sequences associated to the front diagram of the knot.  The number of homotopy classes of augmentations of $(\A(D), \partial)$ is a Legendrian knot invariant.  Currently, it is unknown to the authors how the number of homotopy classes relates to the augmentation numbers considered in this article.

\item[(iv)]  More refined augmentation number invariants may be defined in a similar manner by counting augmentations that take some fixed value on the group ring generators, $t_1, \ldots, t_c$.  The possible values that augmentations may take on group ring generators are studied in the recent article of Leverson, \cite{Leverson2014}.  In particular, when $m$ is even and $L$ is a $1$-component knot, any $m$-graded augmentation must have $\varepsilon(t_1) = -1$. 
\end{itemize}
\end{remark}

\subsection{Proof of Theorem \ref{thm:Main}}

The proof of Theorem \ref{thm:Main} is based upon the following key result concerning the structure of the augmentation variety.

\begin{theorem} \label{thm:structure}
Suppose $\front$ is the nearly plat front diagram of a Legendrian link $\Leg$ with fixed Maslov potential. Then, we can decompose the augmentation variety into a disjoint union
\[
V_m(D, \F) = \coprod_\rho W_\rho
\]
where the union is over all $m$-graded normal rulings of $\front$.  Each subset $W_\rho$ is the image of an injective regular map
\[
\varphi_\rho: (\F^{\times})^{j(\rho) + c} \times \F^{r(\rho)} \hookrightarrow V_m(D, \F)
\]
where $j(\rho) = \#\mbox{switches} - \#\mbox{right cusps}$ and $c$ is the number of components of $L$.  Finally, $r(\rho)$ is the number of $m$-graded returns of $\rho$ when $m \neq 1$ and is the number of $m$-graded returns and right cusps when $m=1$. 
\end{theorem}

Note that this gives the dimension computation\footnote{Equation (\ref{eq:dimC}) is a special case of the following statement.

If $V_1, \ldots, V_n, W$ are complex affine algebraic sets with the $V_i$ non-singular, and there exist injective regular maps 
\[
\varphi_i: V_i \hookrightarrow W \quad \mbox{with} \quad W= \cup \varphi_i(V_i),
\]
then $\dim W = \max \{ \dim V_i \}$.

In the case where $W$ is irreducible and non-singular, this is standard differential topology since all the spaces are smooth manifolds with the maps $\varphi_i$ smooth.  The general case may be treated by appropriately reducing consideration to the non-singular part of a component of $W$ with maximal dimension.
}
\begin{equation} \label{eq:dimC}
\dim_\cc V_m(D, \cc) = \max \{ j(\rho) + c + r(\rho)  \, |\, \rho \mbox{ an $m$-graded ruling of $D$}\}.
\end{equation}

We also use the following lemma that gives a relation between the values of $j(\rho)$ and $r(\rho)$ for distinct normal rulings.

\begin{lemma} \label{lem:returns2} If $\rho$ and $\rho'$ are two $m$-graded normal rulings of the the same front diagram $D$, then
\[
j(\rho) + 2 r(\rho) = j(\rho') + 2 r(\rho').
\]
\end{lemma}

Assuming Theorem \ref{thm:structure} and Lemma \ref{lem:returns2} we now prove Theorem \ref{thm:Main}.

\begin{proof}[Proof of Theorem \ref{thm:Main}]
There is no loss of generality in assuming $L$ is nearly plat since both sides of equation (\ref{eq:11}) are Legendrian isotopy invariants.  According to Equation (\ref{eq:dimC}) we can fix an $m$-graded normal ruling $\rho_0$ so that $\dim_\cc V_m(D, \cc) = j(\rho_0) + c + r(\rho_0)$.  It follows from Lemma \ref{lem:returns2} that $j(\rho_0)$ must be maximal so that $j(\rho_0) = d \left(= \deg_z R^m_L(z)\right)$.  Now, using Theorem \ref{thm:structure} we compute 
\begin{equation} \label{eq:Main1}
\mathit{Aug}_m(L; q)  =  q^{-\dim V_m(D, \cc)}\cdot |V_m(D; \F_q)| =  
 q^{-[d+c+r(\rho_0)]}\sum_{\rho} (q-1)^{j(\rho)+c} q^{r(\rho)}. 
\end{equation} 
For any $m$-graded ruling $\rho$, Lemma \ref{lem:returns2} gives $r(\rho)= r(\rho_0) +\frac{1}{2}( d - j(\rho))$ and making this substitution in the summation allows us to simplify (\ref{eq:Main1}) to
\[
  q^{-[d+c]/2} \sum_{\rho} (q^{1/2}-q^{-1/2})^{j(\rho)+c} = q^{-[d+c]/2}z^{c} R^m_L(z).
\]
\end{proof}

Constructions leading to the proof of Theorem \ref{thm:structure} occupy most of the remainder of the article with the proof completed at the end of Section~\ref{s:MCS2}.  We conclude the present section by proving Lemma \ref{lem:returns2} using the following proposition whose statement and proof are analogous to Lemma 5 of \cite{Ng2006}.

\begin{proposition}
\label{prop:rminusd}
For  $\ruling$  an $m$-graded normal ruling of a nearly plat front diagram $\front$, let $r$ and $d$ denote the number of $m$-graded returns and $m$-graded departures of $\ruling$, respectively.  The difference $r - d$ depends only on $\front$ and its Maslov potential, and, in particular, is independent of the ruling $\ruling$. 
\end{proposition}

\begin{proof}
Suppose that the left cusps of $D$ occur just before $x=x_0$ and the right cusps occur just after $x=x_1$.
Let $E_1, \ldots, E_n$ denote the disks of $\ruling$, and let $M_i$ denote the value of the Maslov potential (mod $m$) on the upper strand of $E_i$.  
  For values of $x \in [x_0,x_1]$ that do not coincide with crossings or cusps of $D$, we assign a value $a_{ij}(x)= 0$ or $1$ to each pair of disks $E_i$ and $E_j$ as follows:  
  
  If $E_{i}$ and $E_{j}$ are disjoint at $x$ with $M_i = M_j +1$  (resp. $M_j = M_i +1$) and $E_i$ is above (resp. below) $E_j$, then $a_{ij}(x)=1$.  If $E_{i}$ and $E_{j}$ are nested at $x$ with $M_i = M_j$, then $a_{ij}(x)=1$.  In all other cases, $a_{ij}(x) = 0$.

Next, set $A(x) = \sum_{i<j} a_{ij}(x)$;  observe that, as $x$ increases, $A(x)$ increases (resp. decreases) by $1$ when passing an $m$-graded returns (resp. $m$-graded departures) and is unchanged by all other types of crossings.  Therefore, by starting at $x=x_0$ and proceeding to $x=x_1$ we see that
\[
r-d = A(x_0) -A(x_1).
\] 
As the right hand side depends only on the Maslov potential of $D$, the Proposition follows.
\end{proof}

\begin{proof}[Proof of Lemma \ref{lem:returns2}]
In order to simplify the notation, we let $s, r,$ and $d$ denote the number of switches, $m$-graded returns, and $m$-graded departures of a normal ruling $\ruling$. We also let $j = s - \gamma$, where $\gamma$ is the number of right cusps of $\front$.

Suppose $\ruling$ and $\ruling'$ are $m$-graded normal rulings of $\front$. If $m\neq1$, we must show $j + 2 r = j' + 2r'$ holds, whereas if $m=1$, we must show $j + 2 (r+\gamma) = j' + 2(r'+\gamma)$ holds. In either case, it suffices to show $j + 2 r = j' + 2r'$. From Proposition~\ref{prop:rminusd}, we have $r - d = r' - d'$.  Moreover, we have the identity $s + r + d = s' + r' + d'$, since this is simply the number of crossings of degree $0$ modulo $m$. Thus,
$s - s' = (r' - r) + (d' - d)$. We use this identity to prove $j + 2 r = j' + 2r'$ holds if and only if $r - d = r' - d'$ holds:  
\begin{eqnarray*}
j + 2 r = j' + 2r' &\Leftrightarrow& s + 2 r = s' + 2 r'  \\ \nonumber
&\Leftrightarrow& s - s' = 2(r'-r) \\ \nonumber
&\Leftrightarrow& (r' - r) + (d' - d) = 2 (r' - r) \\ \nonumber
&\Leftrightarrow& r - d = r' - d' 
\end{eqnarray*}
\end{proof}

%% file: Sections/MCSsSRFormsDiskEquations2.tex
\section{Morse Complex Sequences over $\ring$}
\mylabel{s:MCS1}

To study the augmentation variety using normal rulings, we make use of Morse complex sequences which are combinatorial structures on front diagrams that are more refined than normal rulings.  The definition of a Morse complex sequence is due to Pushkar \cite{Pushkar'a}, and first appears in print in \cite{Henry2011}. The definition is motivated by generating families; see also \cite{Henry2013}.  
As we shall see, Morse complex sequences in certain standard forms, namely SR-form and A-form, are closely related to normal rulings and augmentations, respectively.

Before defining  Morse complex sequences we begin with some preliminaries.  Let $D$ be a front diagram for a Legendrian link, and let $\ring$ be a commutative ring with identity. Let $\ring^\times$ denote the group of units of $\ring$.  A \textbf{handleslide mark} (often just called a handleslide) on the front diagram $\front$ is a pair consisting of $r \in \ring$, called the \textbf{coefficient}, and a vertical line segment in the $xz$-plane avoiding crossings and cusps of $\front$ with endpoints on two strands of $\front$. Given $a < b$, the subset of $D$ with $\{ a \leq x \leq b \}$ is an \textbf{elementary tangle} if it contains a single crossing, left cusp, right cusp, or handleslide mark.

As in the discussion of the Chekanov-Eliashberg algebra, we assume each component of $\front$ includes a marked point $*_i$. Whenever we work with Morse complex sequences, we require the marked points of $\front$ to be located at right cusps and we will call such right cusps {\bf marked}. 

\begin{definition}
\mylabel{defn:MCS}
A \textbf{Morse complex sequence} (abbr. \textbf{MCS}) over $\ring$ for $D$ is a triple $\MCS = ( \{ (C_l, d_l) \}, \{ x_l \}, H )$ where: 

\begin{enumerate}
	\item $H$ is a collection of handleslides marks on $\front$ with coefficients in $\ring$;
	\item $\{ x_l \}$ is an increasing sequence, $x_0 < x_1 < \cdots < x_{M}$, so that all $x$ values of $D$ lie between $x_0$ and $x_M$ and the vertical lines $\{x=x_l\}$ decompose $\front \cup H$ into a collection of elementary tangles;
	\item For each $0 \leq l \leq M$, $(C_l,d_l)$ is an ungraded complex, i.e. $C_l$ is an $R$-module with differential $d_l: C_l \rightarrow C_l$ satisfying $d_l^2=0$.  Moreover, $C_l$ is free with a basis consisting of the points of $\front \cap \{ x = x_l \}$ labeled $e_1, \ldots, e_{s_l}$ from \emph{top to bottom}. 
	The differential $d_l$ is required to be triangular in the sense that for all $1 \leq j \leq s_l$, \[d_l e_j = \sum_{j<k} r_{jk} e_k \text{ where } r_{jk} \in \ring ;\]
	\item For each $0 \leq l \leq M$, if the $k$ and $k+1$ strands at $x=x_l$ meet at a crossing  (resp. left cusp) in a tangle bordered by $x=x_l$ then the coefficient of $e_{k+1}$ in $d_l e_k$, denoted $\langle d_l e_k, e_{k+1} \rangle$, satisfies $\langle d_l e_k, e_{k+1} \rangle = 0$ (resp. $\langle d_l e_k, e_{k+1} \rangle = 1$).  If they meet at an unmarked right cusp, we require $\langle d_l e_k, e_{k+1} \rangle =-1$.  If the right cusp contains the marked point $*_i$, then we require only that $\langle d_l e_k, e_{k+1} \rangle=-s_i$, for some invertible element $s_i \in \ring^{\times}$ and we say that $\MCS$ \textbf{assigns the value $s_i$} to the marked point $*_i$.  
	\item 
Finally, for $0\leq l < M$ the complexes $(C_l,d_l)$ and $(C_{l+1},d_{l+1})$ are required to satisfy some relationship depending on the tangle $T$ appearing between $x_l$ and $x_{l+1}$.

\begin{enumerate}
	\item If $T$ contains a \emph{crossing} between strands $k$ and $k+1$, then $\varphi : (C_l, d_l) \to (C_{l+1}, d_{l+1})$ defined by 
$$
\varphi(e_i) = \left\{ \begin{array}{rl}
 e_i &\mbox{ if $i \notin \{ k, k+1 \} $} \\
 e_{k+1} &\mbox{ if $i = k$ } \\
 e_{k} &\mbox{ if $i = k+1$ } 
       \end{array} \right.
$$
is an isomorphism of complexes;
	\item If $T$ contains a \emph{handleslide} between strands $j$ and $k$ with $j < k$ with coefficient $r \in \ring$, then $\varphi : (C_l, d_l) \to (C_{l+1}, d_{l+1})$ defined by 
$$
\varphi(e_i) = \left\{ \begin{array}{rl}
 e_i &\mbox{ if $i \neq j $} \\
 e_{j}-r e_k &\mbox{ if $i = j$ } 
       \end{array} \right.
$$
is an isomorphism of complexes;
	\item If $T$ contains a \emph{right cusp} between strands $k$ and $k+1$, then the linear map 
$$
\varphi(e_i) = \left\{ \begin{array}{rl}
 [e_i] &\mbox{ if $i < k $} \\
 $[$e_{i+2}] &\mbox{ if $i \geq k$ } 
       \end{array} \right.
$$
is an isomorphism of complexes from $(C_{l+1}, d_{l+1})$ to the quotient of $(C_{l}, d_l)$ by the acyclic subcomplex spanned by $\{e_k, d_{l} e_k \}$. 

	\item If $T$ contains a \emph{left cusp} between strands $k$ and $k+1$, then $(C_{l}, d_{l})$ and $(C_{l+1}, d_{l+1})$ are related as in the case of a right cusp with the roles of $(C_{l}, d_{l})$ and $(C_{l+1}, d_{l+1})$ interchanged. 
	
\end{enumerate}
\end{enumerate}
\end{definition}

Just as in the case of normal rulings and augmentations, the notion of an MCS may be refined once we choose a Maslov potential for $D$ and a divisor $m$ of $2r(L)$.  Such a choice provides each $C_l$ with a $\Z/m\Z$-grading, $C_l = \bigoplus_{ i \in \Zm} (C_l)_i$, where $(C_l)_i$ is the span of those strands at $x=x_l$ for which the Maslov potential is congruent to $i \mod m$.  We say that an MCS $\MCS$ is \textbf{$m$-graded} if
\begin{enumerate}
\item for all $l$, the differential $d_l$ has degree $-1 \mod m$, i.e. $d_l(C_l)_i \subset (C_l)_{i-1}$ and
\item all handleslide marks have their endpoints on strands for which the Maslov potential takes the same value modulo $m$. 
\end{enumerate}

Note that, in general, the complexes $(C_l, d_l)$ are not determined by the handleslide set since if a left cusp lies between $x_l$ and $x_{l+1}$, then (5) (d) of Definition \ref{defn:MCS} does not allow us to uniquely recover $(C_{l+1},d_{l+1})$ from $(C_l,d_l)$.  A restriction at left cusps may be imposed to overcome this problem.  For an MCS $\MCS$, we say that 
a left (resp. right) cusp between strands $k$ and $k+1$ in the tangle to the left (resp. right) of $x_l$ is \textbf{simple} if $ d_{l} e_{k}= t e_{k+1}$ with  $t \in \field^{\times}$ as in Definition \ref{defn:MCS} (4), and $\langle d_{l+1} e_j, e_k \rangle = \langle d_{l+1} e_j, e_{k+1} \rangle = 0$ for all $j < k$. 

\begin{proposition} \label{prop:simple}
 An MCS with simple left cusps is uniquely determined by its handleslide set $H$.
\end{proposition}
\begin{proof}
To show that the sequence of complexes, $(C_l,d_l)$, is uniquely determined by $H$ work inductively from $x_0$ to $x_M$.  At $x_0$, the generating set is empty, so $C_0 = \{0\}$.  Once $(C_l,d_l)$ is known, $(C_{l+1},d_{l+1})$ is uniquely determined by Definition \ref{defn:MCS} (5) (a), (b), or (c) in the case the tangle between $x_l$ and $x_{l+1}$ contains a crossing, handleslide, or right cusp, respectively.  If a left cusp sits between $x_l$ and $x_{l+1}$, then combining the assumption that the left cusp is simple with Definition \ref{defn:MCS} (4) and (5) (d) shows that  $(C_{l+1},d_{l+1})$ is isomorphic to the split extension of complexes $(C_{l},d_{l}) \oplus ( \mbox{Span}_R\{a,b\}, d')$ where $d' a = b$; the isomorphism  is induced by mapping the generators of $C_l$, i.e strands of $D$ at $x_l$, to the corresponding strands at $x_{l+1}$ and mapping $a$ and $b$ to the upper and lower strands of the cusp, respectively.
\end{proof}

 Given an arbitrary handleslide set $H$, we may attempt to find a sequence of complexes $\{(C_l,d_l)\}$ so that $(\{(C_l,d_l)\}, \{x_l\}, H)$ is an MCS with simple left cusps by beginning with $C_0= \{0\}$ and uniquely extending to the right as in the proof of Proposition \ref{prop:simple}.  Note that it is possible that the requirements in Definition \ref{defn:MCS} (4) may not hold when we reach a crossing or right cusp.  However, it is clear from the definition that this is all that can go wrong, so we have the following:

\begin{proposition} \label{prop:H}
A handleslide set $H$ is the handleslide set of an MCS with simple left cusps if and only if when inductively defining the complexes $\{(C_l,d_l)\}$ from left to right we have $\langle d_l e_k, e_{k+1} \rangle = 0$; $\langle d_l e_k, e_{k+1} \rangle = -1$; or $\langle d_l e_k, e_{k+1} \rangle \in R^{\times}$ whenever $x_l$ respectively precedes a crossing; unmarked right cusp; or marked right cusp at which the $k$ and $k+1$ strands meet. 
\end{proposition}

\begin{figure}[t]
\labellist
\small\hair 2pt
\pinlabel {$a$} [br] at 18 78
\pinlabel {$b$} [br] at 31 31

\pinlabel {$r$} [br] at 48 54

\pinlabel {$a$} [br] at 66 78
\pinlabel {$-ar$} [br] at 90 78
\pinlabel {$rb$} [br] at 94 54
\pinlabel {$b$} [br] at 104 31

\pinlabel {$a$} [br] at 166 78
\pinlabel {$b$} [br] at 185 54
\pinlabel {$r$} [br] at 206 31
\pinlabel {$a$} [br] at 226 78
\pinlabel {$-br$} [br] at 249 54
\pinlabel {$b$} [br] at 266 54
\pinlabel {$a$} [br] at 329 78
\pinlabel {$-br$} [br] at 353 54
\pinlabel {$b$} [br] at 368 54
\pinlabel {$-r^{-1}$} [br] at 401 31
\pinlabel {$a$} [br] at 410 78
\pinlabel {$ar^{-1}$} [br] at 443 78
\pinlabel {$-br$} [br] at 454 54
\pinlabel {$ar^{-1}b^{-1}$} [br] at 491 78
\pinlabel {$ar^{-1}$} [br] at 525 78
\pinlabel {$-br$} [br] at 496 54
\endlabellist
\centering
\includegraphics[scale=.9]{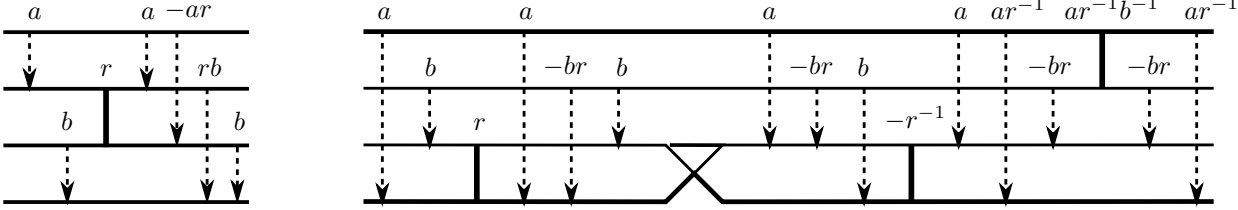}
\caption{We use dotted arrows to encode the differentials of chain complexes of $\MCS$. If $\langle d e_i, e_j \rangle \neq 0$, then a dotted arrow is drawn from strand $i$ to strand $j$ with label $\langle d e_i, e_j \rangle$. The left figure demonstrates the effect of a handleslide on the differential of a chain complex. The right figure shows the differentials of chain complexes near a type (S2) switch in an SR-form MCS.}
\label{fig:switchgradients}
\end{figure}

\begin{figure}[t]
\labellist
\small\hair 2pt
\pinlabel {$a$} [br] at 17 165
\pinlabel {$b$} [br] at 17 117
\pinlabel {$r$} [br] at 50 140
\pinlabel {$-r^{-1}$} [bl] at 117 140
\pinlabel {$-ar$} [bl] at 151 165
\pinlabel {$br$} [bl] at 151 117

\pinlabel {$a$} [br] at 204 142
\pinlabel {$b$} [bl] at 222 142
\pinlabel {$r$} [br] at 242 117
\pinlabel {$-r^{-1}$} [bl] at 305 117
\pinlabel {$ar^{-1}b^{-1}$} [br] at 317 165
\pinlabel {$ar^{-1}$} [bl] at 349 142
\pinlabel {$-br$} [br] at 334 142

\pinlabel {$a$} [br] at 398 142
\pinlabel {$b$} [bl] at 415 142
\pinlabel {$r$} [br] at 436 165
\pinlabel {$-r^{-1}$} [bl] at 498 165
\pinlabel {$ar^{-1}b^{-1}$} [br] at 508 116
\pinlabel {$-r^{-1}a$} [bl] at 539 142
\pinlabel {$rb$} [br] at 523 142

\pinlabel {$a$} [br] at 15 53
\pinlabel {$b$} [bl] at 31 6
\pinlabel {$r$} [br] at 50 29
\pinlabel {$a$} [br] at 136 52
\pinlabel {$b$} [br] at 136 6

\pinlabel {$a$} [br] at 204 53
\pinlabel {$b$} [bl] at 221 30
\pinlabel {$r$} [br] at 245 6
\pinlabel {$arb^{-1}$} [br] at 300 52
\pinlabel {$b$} [br] at 330 30
\pinlabel {$a$} [bl] at 348 53

\pinlabel {$a$} [br] at 396 53
\pinlabel {$b$} [bl] at 413 30
\pinlabel {$r$} [br] at 436 52
\pinlabel {$a^{-1}rb$} [br] at 492 6
\pinlabel {$a$} [br] at 523 30
\pinlabel {$b$} [bl] at 539 53

\pinlabel {(S1)} [tl] at 70 110
\pinlabel {(S2)} [tl] at 261 110
\pinlabel {(S3)} [tl] at 456 110

\pinlabel {(R1)} [tl] at 70 0
\pinlabel {(R2)} [tl] at 261 0
\pinlabel {(R3)} [tl] at 456 0

\endlabellist
\centering
\includegraphics[scale=.65]{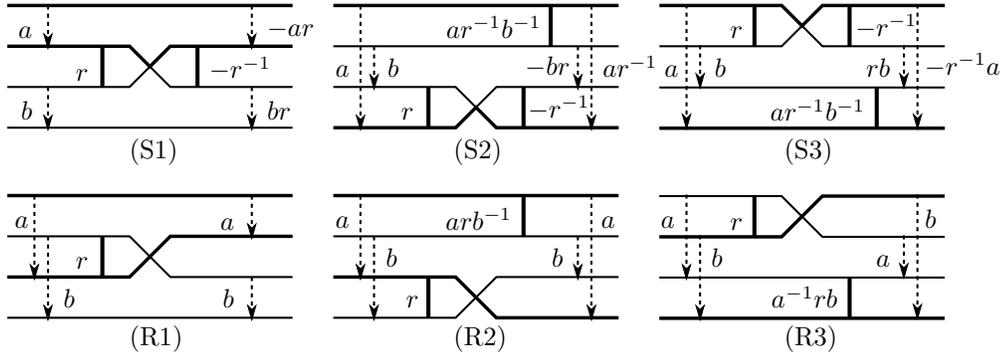}
\caption{The arrangement of handleslides near switches (top row) and returns (bottom row) in an SR-form MCS compatible with a normal ruling. For returns we allow $r=0$. Note that if the chain complex to the left of the left-most handleslide is standard with respect to a normal ruling, then the chain complex to the right of the right-most handleslide is as well.}
\label{fig:SRform}
\end{figure}

\subsection{SR-form MCSs and normal rulings}
In this subsection we will define a type of MCS, namely an SR-form MCS, that has a standard form with respect to a ruling $\rho$.  Such an MCS has simple left cusps and handleslides arranged in a particular manner near switches and $m$-graded  returns of $\ruling$.  Consequently, for many values of $l$ the chain complexes $(C_l,d_l)$ will be closely related to the involutions associated with $\ruling$;  see Lemma \ref{lem:NormalForm} below.  In the remainder of this section, we restrict attention to MCSs with coefficients in a field $\field$.

Let $\ruling$ be a normal ruling of a front diagram $\front$, and let $\MCS = ( \{ (C_l, d_l) \}, \{ x_l \}, H )$ be an MCS for $D$.  We denote by $\ruling_{x_l}$  the fixed-point free involution of $\ruling$ on the points $D \cap (\{x_l\}\times \R)$. We say the chain complex $(C_l, d_l)$ of $\MCS$ is \textbf{standard with respect to $\ruling$} if for all $i < j$ $$\langle d_l e_i, e_j \rangle \neq 0 \mbox{ if and only if } \ruling_{x_l}(i)=j.$$  (In our pictorial presentation of MCSs and normal rulings, as in Figure \ref{fig:switchgradients} and Section 2.1.1,  at $x=x_l$ dotted arrows connect precisely those strands that form the boundary of a common ruling disk.)

Suppose, in addition, that $D$ is equipped with a Maslov potential with respect to which $\rho$ is $m$-graded.
\begin{definition}
\mylabel{defn:SR-form}
We say an $m$-graded MCS, $\MCS = ( \{ (C_l, d_l) \}, \{ x_l \}, H )$, over a field $\F$ is in \textbf{SR-form compatible with $\ruling$} if every left cusp is simple and the handleslide set of $\MCS$ consists of only the following handleslides:

\begin{enumerate}
	\item Near \textit{switches}:  The handleslide marks and their coefficients appear as in the top row of Figure \ref{fig:SRform}.  Specifically, there are handleslides connecting the crossing strands immediately to the left and right of the crossing with respective coefficients $r$ and $-r^{-1}$ for some $r \in \F^{\times}$.  For a switch of type (S2) or (S3) these handleslides are immediately followed by a third handleslide that connects the companion strands of the switch strands with coefficient $ar^{-1}b^{-1}$.  Here, $a$ (resp. $b$) is the coefficient relating the boundary strands of the outer (resp. inner) of the two ruling disks prior to these handleslides.   

	\item Near \textit{$m$-graded returns}: The handleslide marks and their coefficients appear as in the bottom row of Figure \ref{fig:SRform}.   That is, there is always a handleslide connecting the crossing strands prior to the return with coefficient $r \in \F$.  (Here, $r$ is allowed to be $0$.)  For returns of type (R2) and (R3) a second handleslide located after the return connects the two companion strands of the crossing strands with coefficient $arb^{-1}$ and $a^{-1}rb$, respectively. Here, $a$ (resp. $b$) is the coefficient in the differential that relates the boundary strands of the upper (resp. lower) of the two disks prior to these handleslides.   
	
 \item Near \textit{right cusps}:  If $m =1$, there is a handleslide immediately to the left of each right cusp that connects the two strands that meet at this right cusp.  The coefficient $r \in \F$ is arbitrary (and possibly zero).
\end{enumerate}
\end{definition}

The set $\mathit{MCS}_m^{\ruling}(\front;\field)$ consists of all SR-form MCSs of $D$ that are compatible with $\ruling$. We define $\mathit{MCS}_m^{SR}(\front;\field)$ to be $\bigcup_{\ruling} \mathit{MCS}_m^{\ruling}(\front;\field)$, where the union is over all $m$-graded normal rulings of $\front$.

We have the following result concerning the form of the chain complexes in an SR-form MCS.

\begin{lemma} \label{lem:NormalForm}
Except for $x_l$ within the clusters of handleslides near switches and returns (as specified by Definition \ref{defn:SR-form}), the chain complexes, $(C_l,d_l)$, of an SR-form MCS, $\MCS$, associated with a graded normal ruling $\rho$ are standard with respect to $\rho$.
\end{lemma}

\begin{proof} Working left to right, we assume that the complex $(C_l,d_l)$ has standard form with respect to $\rho$ before a cusp, or crossing and verify that this is still the case after passing the cusp or crossing (and any corresponding handleslides).  For cusps, this follows from the assumption that left cusps are simple and Definition \ref{defn:MCS} (5c) and (5d).  When passing a crossing that is a departure, requirement (5a) in Definition \ref{defn:MCS} implies that the corresponding complexes of $\MCS$ change only by a reordering of basis vectors.  The involution in the definition of $\rho$ changes in a corresponding way at a departure, so that the complex $(C_{l+1},d_{l+1})$ is also in standard form with respect to $\rho$ after the crossing.  In the case  of a switch or return, the way that the collection of handleslides and the crossing affect the differential is specified in Definition \ref{defn:MCS} (5a) and (5b).    Note that for the case of an (S2) switch, a detailed calculation is carried out in Figure \ref{fig:switchgradients} where the coefficients $\langle d_l x_i, x_j \rangle$ are indicated with dotted arrows from the $i$-th strand to the $j$-th strand at $x=x_l$.  The results of the calculation in general are shown in Figure \ref{fig:SRform}.
\end{proof}

Note that Lemma~\ref{lem:NormalForm} holds in a slightly more general setting.  Suppose $H$ is a handleslide set satisfying the SR-form conditions with respect to $\rho$.  If the conditions of Proposition \ref{prop:H} are satisfied up to $x_l$ so that $(C_l,d_l)$ may be defined and $x_l$ is not located between the handleslides of a switch or return, then $(C_l,d_l)$ is standard with respect to $\rho$.  In Section \ref{s:MCS2}, we use a further generalization where the normal ruling $\rho$ is itself only defined for the portion of the front diagram to the left of $x=x_l$.  Either of these extensions follow from the same proof as Lemma \ref{lem:NormalForm} since the induction proceeds from left to right. 

\subsection{SR-form Morse complex sequences as solutions of a system of equations} \label{sec:diskEquation}

All left cusps of an SR-form MCS are simple, so Proposition \ref{prop:simple} implies an SR-form MCS is uniquely determined by the collection of handleslide marks, $H$.  
Moreover, a potential collection of handleslide marks, $H$, for an SR-form MCS associated to a normal ruling, $\rho$, is determined by a choice of non-zero coefficient $r \in \F^\times$ at each of the switches of $\rho$ and a (possibly zero) coefficient $r \in \F$ at each of the $m$-graded returns of $\rho$ and, if $m=1$, at right cusps.  Here, we take $r$  to be the coefficient of the left-most handleslide of those handleslides grouped near the switch or return;  the coefficients of the remaining handleslides at the  switch or return are then specified as in Definition \ref{defn:SR-form} once the complexes $(C_l,d_l)$ have been determined before the switch or return.   (Note that as long as the handleslide set produces an MCS prior to the given switch or return, the coefficients $a$ and $b$ must be non-zero according to Lemma \ref{lem:NormalForm} and there is no problem with inverting them.)

Given such a handleslide set, $H$, we may apply Proposition \ref{prop:H} to determine if $H$ corresponds to an SR-form MCS.  To this end, we claim that the condition 
\begin{equation} \label{eq:CrRestrict}
\langle d_l e_k, e_{k+1} \rangle = 0
\end{equation}  
is always satisfied prior to a crossing between the $k$ and $k+1$ strands.  To see this, apply the generalization of Lemma \ref{lem:NormalForm}.  Strands that meet at a crossing cannot be paired by the ruling at the crossing, and hence the required coefficient is zero before the crossing.  (In the case of a switch or return, Lemma~\ref{lem:NormalForm} tells us $\langle d_l e_k, e_{k+1} \rangle = 0$ before the handleslide that precedes the crossing.  However, passing this handleslide will not affect the vanishing of the coefficient in question.)

Although (\ref{eq:CrRestrict}) always holds before crossings, the requirements of Proposition \ref{prop:H} may fail when we reach a right cusp.
Indeed, Lemma \ref{lem:NormalForm} combined with the fact that strands meeting at a right cusp must be paired by $\rho$ shows that the coefficient $\langle d_l e_k, e_{k+1} \rangle$ is non-zero, but it is not necessarily $-1$ as required if the cusp is not marked.  

For a given right cusp, the coefficient $\langle d_l e_k, e_{k+1} \rangle$ may be computed as follows.  Recall from Section~\ref{sec:disks}, the right cusp forms the right endpoint of a disk, $D_i$, of the ruling, $\rho$. We keep track of the coefficient, $a$, with which the generator corresponding to the  lower strand of this disk appears in the differential of the upper strand as we move from the left cusp of the disk towards this right cusp.  Initially, the coefficient is $1$, as this is required in (4) of Definition \ref{defn:MCS}.  The coefficient, $a$, remains constant except when passing the collection of handleslides near a switch or return involving $D_i$ and another disk of $\rho$.  The overall effect of these handleslides on $a$ is indicated in Figure \ref{fig:SRform}; passing a return or right cusp does not affect $a$ at all, while passing a switch replaces $a$ with a multiple $c \, a$.  The factor $c$ is determined by the first handleslide coefficient, $r$, of the switch, and the combinatorics of $D_i$ at the switch as follows.  
\begin{enumerate}
\item When the switch is Type (S1) and $D_i$ is the upper (resp. lower) of the two disks $c = -r$ (resp. $c = r$).
\item When the switch is Type (S2) and $D_i$ is the inner (resp. outer) of the two disks $c = -r$ (resp. $c = r^{-1}$).  
\item When the switch is Type (S3) and $D_i$ is the inner (resp. outer) of the two disks $c = r$ (resp. $c = -r^{-1}$).  
\end{enumerate}
These calculations are indicated in the top line of Figure \ref{fig:SRform} with an expanded calculation for a Type (S2) switch in Figure \ref{fig:switchgradients}.

When the right cusp of $D_i$ is reached, the coefficient $\langle d_l e_k, e_{k+1} \rangle$ is the product of all the factors, $c$, associated to $D_i$ by switches appearing along the boundary of $D_i$.  Therefore, according to Proposition \ref{prop:H} a collection of handleslide marks $H$ arranged as in the definition of an SR-form MCS corresponds to an MCS if and only if this product is equal to $-1$ for every disk $D_i$ that does not contain a marked point on its right cusp.

Recall that we have denoted the set of all SR-form MCSs compatible with a fixed normal ruling $\rho$ by $\mathit{MCS}_m^{\ruling}(\front;\field)$.  
The above discussion allows us to realize $\mathit{MCS}_m^{\ruling}(\front;\field)$ as the solution set of a system of polynomial equations, i.e. an affine algebraic set.  For this purpose, we introduce variables, $x_1, \ldots, x_n$, corresponding to the switches, $s_1, \ldots, s_n$, of $\rho$.  To each disk of the ruling, $D_i$, we associate an equation
\begin{equation} \tag{$R_i$}
\prod_j y_j = w_i
\end{equation}   
where the product is over those $j$ such that the switch, $s_j$, corresponding to $x_j$ involves the disk $D_i$; the right hand side is
\[
w_i = \left\{ \begin{array}{ll} -1 & \mbox{if $D_i$ does not have a marked point at its right cusp;} \\
																-t_k & \mbox{if $D_i$ has the marked point $*_k$ at its right cusp;} \\
																\end{array} \right.
\]
and the factors $y_j$ are given by
\[
y_j = \left\{ \begin{array}{ll} -x_j & \mbox{if $s_j$ has Type $(S1)$ and $D_i$ is the upper disk;} \\
																x_j & \mbox{if  $s_j$ has Type $(S1)$ and $D_i$ is the lower disk;} \\
																-x_j & \mbox{if $s_j$ has Type $(S2)$ or $(S3)$ and $D_i$ is the inner disk;} \\
																x_j^{-1} & \mbox{if $s_j$ has Type $(S2)$ or $(S3)$ and $D_i$ is the outer disk.} \\ 
																\end{array} \right.
																\]
We refer to the system of equations $(R_i)$ as the \textbf{disk equations} associated to the ruling $\rho$.

\begin{theorem} \label{thm:zrho} The set $\mathit{MCS}_m^{\ruling}(\front;\field)$ is in bijection with the affine algebraic set  
\[
Z_\rho \subset (\F^\times)^c \times (\F^\times)^n \times \F^{r} = \{ (t_1, \ldots, t_c, x_1, \ldots, x_n, z_1, \ldots, z_r)\}
\]
 given by the solution set of the disk equations $(R_i)$ associated to $\rho$.  Here, $c$ and  $n$ denote the number of components of $L$ and the number of switches.  If $m \neq 1$ (resp. $m=1$),  $r$ denotes the number of $m$-graded returns of $\rho$ (resp. number of $m$-graded returns and right cusps).    
\end{theorem}
\begin{proof}
Let values of $(t_1, \ldots, t_c, x_1, \ldots, x_n, z_1, \ldots, z_r) \in (\F^\times)^c \times (\F^\times)^n \times \F^{r}$ be given.  Create a collection of handleslide marks $H$ arranged as in an SR-form MCS compatible with $\rho$  by taking the first coefficient $r$ at the switch $s_j$ to be $x_j$  except when $s_j$ has Type $(S3)$, in which case we take $r = -x_j$.  In addition, at returns and right cusps we use $z_j$ for the first handleslide coefficient.  Then, the above discussion shows that $H$ forms the handleslide set of an SR-form MCS with values $(t_1, \ldots t_c)$ assigned to marked points if and only if the disk equations $(R_i)$ hold.
\end{proof}

\subsection{Disk equations associated to ruling graphs}
In this subsection, we formalize the combinatorial data from  a normal ruling that is needed to formulate the disk equations by introducing the notion of a ruling graph.  Ruling graphs provide a slightly wider setting for considering the disk equations, and this allows us to, in the remainder of this section, analyze the varieties $Z_\rho$ via an inductive approach.  

\begin{definition} A {\bf ruling graph} is an abstract graph $\Gamma$ with vertex set $V = \{v_1, \ldots, v_N\}$ and edge set $E = \{e_1, \ldots, e_n\}$ with the following additional structure:
\begin{itemize}
\item[(1)] Each vertex $v_i$ is labeled with an invertible element $w_i \in \F[t_1^{\pm 1}, \ldots, t_c^{\pm 1}]$;
\item[(2)] Edges are oriented; and
\item[(3)] There is a function $\alpha : E \rightarrow \{D, N\}$ which provides each edge with a type of $D$ or $N$.
\end{itemize}
\end{definition}

Each ruling graph has an associated system of equations, $P(\Gamma)$, which we call the {\bf disk equations}.  The  variables, $x_1, \ldots, x_n$, are in correspondence with the edges of $\Gamma$.  For each vertex, $v_i$, we define an equation $P_i$ via
\begin{equation}
\tag{$P_i$}
\prod{ y_j} = w_i
\end{equation} 
where the product is over those $j$ such that the edge $e_j$ has an end at $v_i$.  The term $y_j$ is given by
\[
y_j = \left\{ \begin{array}{ll} -x_j & \mbox{if $e_j$ is oriented away from $v_i$} \\
																x_j & \mbox{if $e_j$ is oriented towards $v_i$ and $\alpha(e_j) = D$} \\
																x_j^{-1} & \mbox{if $e_j$ is oriented towards $v_i$ and $\alpha(e_j) = N$} \\
																-x_j^2 & \mbox{if $e_j$ is a loop at $v_i$ of either type.} \\ 
																\end{array} \right.
																\]
																See Figure \ref{fig:RulGraph} for an example of a ruling graph with its associated disk equations.
																
\begin{figure}
\labellist
\small
\pinlabel $N$ [br] at 46 298
\pinlabel $N$ [b] at 80 245
\pinlabel $D$ [tr] at 78 92
\pinlabel $D$ [b] at 82 168
\pinlabel $N$ [tl] at 252 92
\pinlabel $D$ [tl] at 166 92
\pinlabel $D$ [b] at 244 168
\pinlabel $D$ [b] at 244 206
\pinlabel $D$ [b] at 244 126
\pinlabel $v_1$ [l] at 170 324
\pinlabel $v_2$ [r] at -2 164
\pinlabel $v_3$ [t] at 164 -2
\pinlabel $v_4$ [b] at 328 164
\pinlabel $v_5$ [br] at 158 168
\endlabellist
\centerline{\includegraphics[scale=.5]{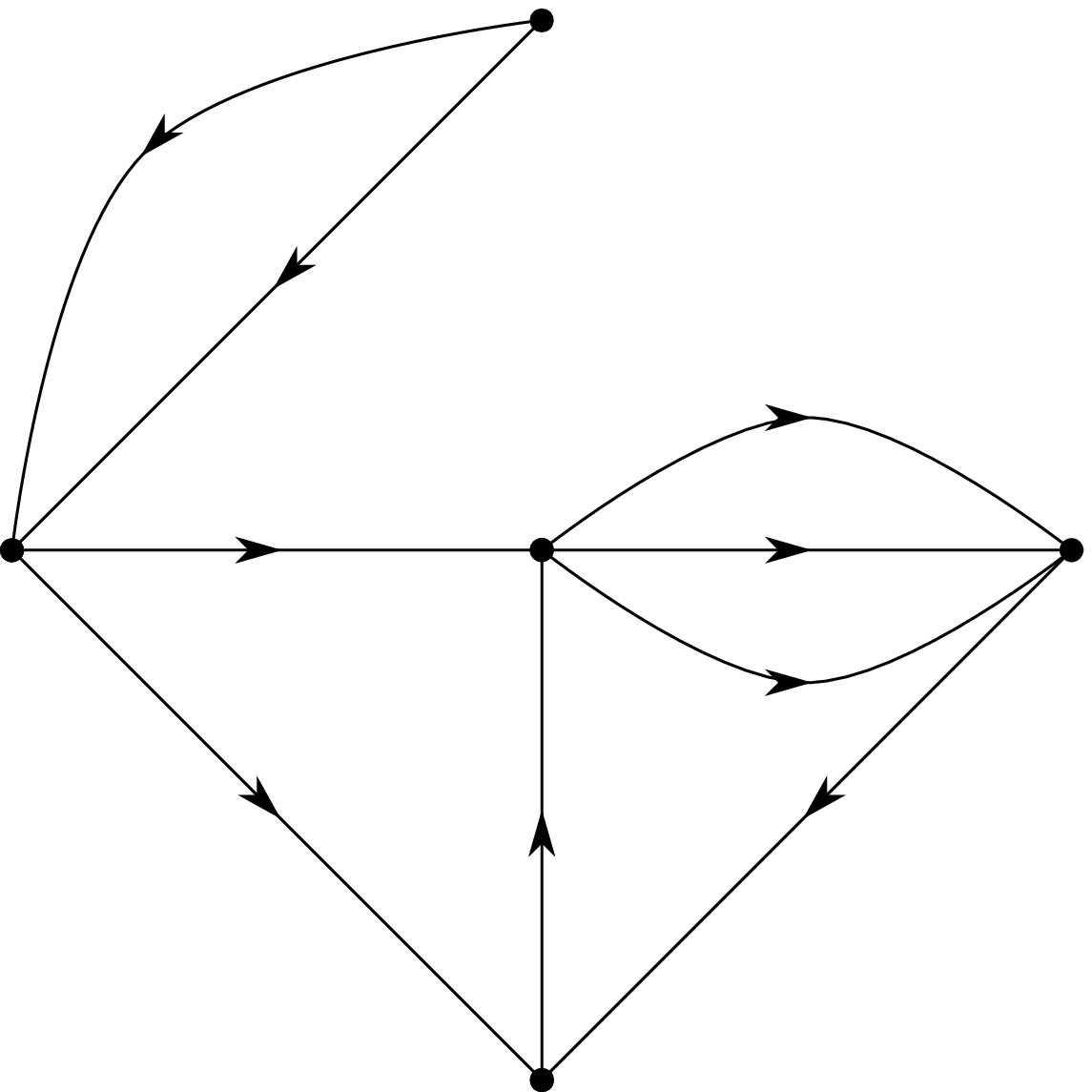} \quad \quad \raisebox{3cm}{$\begin{array}{rl}
(-x_1)(-x_2) & = w_1 \\ &\\
x_1^{-1} x_2^{-1} (-x_3) (-x_4) & = w_2 \\ & \\
x_4 (-x_5) x_6^{-1} & = w_3 \\ & \\
(-x_6)x_7 x_8 x_9 & = w_4 \\ & \\
x_3 x_5 (-x_7)(-x_8)(-x_9) & = w_5
\end{array}$}}
\caption{A ruling graph (left) and its associated disk equations (right) with respect to a particular ordering of the edges.
}
\label{fig:RulGraph}
\end{figure}

\subsubsection{Ruling graphs associated to normal rulings}

We can associate a ruling graph $\Gamma_{\rho}$ to a front diagram $D$ with normal ruling $\rho$ and  base points located at one right cusp of each component.   The vertices of $\Gamma_{\rho}$, $v_1, \ldots, v_N$, are in correspondence with the disks of $\rho$.  The labels of vertices, $w_1, \ldots, w_N$, are given by $w_i = -1$ if the right cusp of $D_i$ is not marked and $w_i = -t_k$ if the right cusp contains the marked point $*_k$. For each switch of $\rho$ we assign an edge between the vertices corresponding to the two disks of the switch.  Moreover, the edge has type $D$ (resp. type $N$) if the two disks are disjoint (resp. nested) at the switch, i.e. if the switch has Type $(S1)$ (resp. Type $(S2)$ or $(S3)$).  We orient edges so that at disjoint (resp. nested) switches the orientation points from the upper disk to the lower disk (resp. from the inner disk to the outer disk).   

\begin{remark}
Not every ruling graph arises from this construction.  For instance, note that $\Gamma_{\rho}$ does not have any loop edges since each switch involves two distinct disks of $\rho$.  
\end{remark}

Lemma~\ref{lem:rg} is an immediate consequence of the definitions.
\begin{lemma} \label{lem:rg}
The disk equations $(P_i)$ associated to the ruling graph $\Gamma_{\rho}$ are identical to the disk equations $(R_i)$ associated to the normal ruling $\rho$.  
\end{lemma}
There is, however, one small distinction at the level of solution sets, since we do not include any variables corresponding to returns or right cusps when considering disk equations of $\Gamma_\rho$.  These variables do not actually appear in the disk equations $(R_i)$ of $\rho$, so it follows that the respective solution sets of $(P_i)$ and $(R_i)$,  are related by
\begin{equation} \label{eq:solutions}
Z_\rho = Z({\Gamma_\rho}) \times \F^r
\end{equation}
with $r$ as in Theorem \ref{thm:zrho}.

\subsubsection{Analyzing the solution set of the disk equations}

For an arbitrary ruling graph $\Gamma$, we let $Z(\Gamma) \subset (\F^{\times})^c\times(\F^{\times})^n$ denote the solution set of the disk equations, $P(\Gamma)$.

\begin{lemma} \label{lem:orient} The isomorphism type of $Z(\Gamma)$ is independent of the orientation of edges.
\end{lemma}
\begin{proof}
Reversing a non-loop edge $e_j$ results in a change of variables for the associated disk equations where all occurrences of $x_j$ are replaced with $-x_j$ if the edge has type $D$ or with $-x^{-1}_j$ if the edge has type $N$.  Reversing a loop edge does not change the disk equations.
\end{proof}

Suppose now that $\Gamma$ is a ruling graph with an edge $e_k$ oriented from $v_i$ to $v_j$ with $v_i \neq v_j$.  A second ruling graph $\Gamma'$ arises from the following procedure:
\begin{enumerate}
\item Remove all edges between $v_i$ and $v_j$ that have the same type ($D$ or $N$) as $e_k$.  Let $s$ denote the number of such edges including $e_k$.
\item Merge $v_i$ and $v_j$ into a single vertex $\tilde{v}$ with label $\tilde{w} = 
\left\{ \begin{array}{ll} (-1)^sw_iw_j & \mbox{if $\alpha(e_k) = N$} \\
													(-1)^sw_i^{-1}w_j	& \mbox{if $\alpha(e_k) = D$}
													\end{array} \right.$.
\item If $e_k$ is of type $D$, change the type of all edges that had previously had exactly one vertex at $v_i$.																									\end{enumerate}
We say that $\Gamma'$ is obtained from $\Gamma$ by {\bf contraction} along the edge $e_k$; see Figure \ref{fig:co}

\begin{figure}
\labellist
\small
\pinlabel $D$ [l] at 48 167
\pinlabel $D$ [r] at -1 167
\pinlabel $N$ [l] at 104 167
\pinlabel $N$ [tr] at 25 287
\pinlabel $D$ [tl] at 95 282
\pinlabel $N$ [bl] at 45 300
\pinlabel $N$ [br] at 41 40
\pinlabel $D$ [bl] at 72 34
\pinlabel $e_k$ [r] at 8 128
\pinlabel $v_j$ [l] at 67 79
\pinlabel $v_i$ [l] at 67 241
\pinlabel $\tilde{v}$ [r] at 262 160
\pinlabel $D$ [tr] at 233 207
\pinlabel $N$ [tl] at 315 215
\pinlabel $D$ [bl] at 253 220
\pinlabel $D$ [tl] at 316 142
\pinlabel $N$ [br] at 245 109
\pinlabel $D$ [bl] at 284 88
\endlabellist
\centerline{\includegraphics[scale=.5]{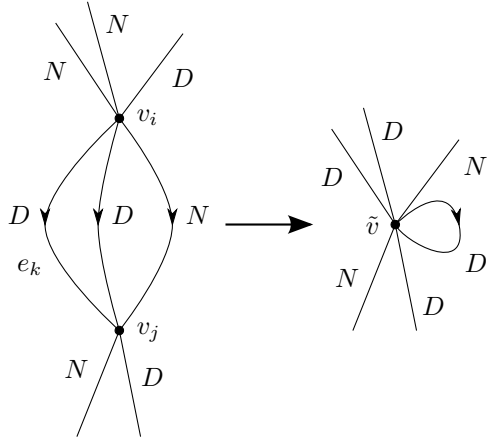} }
\caption{Contracting a ruling graph along the edge $e_k$.
}
\label{fig:co}
\end{figure}

\begin{lemma} \label{lem:contract}  If $\Gamma'$ is obtained from $\Gamma$ by contraction along an edge with $s$ edges deleted during the contraction, then
\[
Z(\Gamma) \cong (\F^\times)^{s-1}\times Z(\Gamma')
\]
as affine varieties.
\end{lemma}
\begin{proof}  In view of Lemma \ref{lem:orient}, we may assume that all edges of $\Gamma$ connecting $v_i$ to $v_j$ are oriented from $v_i$ to $v_j$, and also that any other edges with a single boundary point on $v_i$ or $v_j$ are oriented into $v_i$ or into $v_j$, respectively.  For ease of notation, assume the edges from $v_i$ to $v_j$ are labeled $e_1, \ldots, e_l$ with the contraction along $e_1$.    

\begin{itemize}
\item[Case 1:]  Edges $e_1, \ldots, e_s$ have type $D$, and edges $e_{s+1}, \ldots, e_l$ have type $N$ where $1 < s \leq l$.  

The disk equations corresponding to the edges $v_i$ and $v_j$ have the form 
\begin{equation} \tag{$P_i$}
 (-x_1)\cdots(-x_l) a = w_i,  \mbox{ and }  
\end{equation}
\begin{equation}
\tag{$P_j$} (x_1 \cdots x_s) (x_{s+1} \cdots x_l)^{-1} b = w_j,    
\end{equation}
where $a$ and $b$ denote the product of the remaining terms corresponding to edges with endpoints at precisely one of $v_i$ and $v_j$. 

In this case, we solve for $x_1$ in ($P_i$) and substitute the result into ($P_j$) to produce equivalent equations ($P_i'$) and ($P_j'$), respectively:
\begin{equation} \tag{$P_i'$}
x_1 = (-1)(-x_2^{-1}) \cdots (-x_s^{-1}) (-x_{s+1}^{-1})\cdots (-x_{l}^{-1}) a^{-1} w_i.
\end{equation}
\begin{equation} \tag{$P_j'$}
(-x_{s+1}^{-2}) \cdots (-x_l^{-2}) a^{-1} b = (-1)^s w_i^{-1} w_j
\end{equation}
After contracting the edge $e_1$, $v_i$ and $v_j$ are merged into a single vertex whose disk equation is $P_j'$ with the change of variable $\tilde{x}_j = x_j^{-1}$ for  $s+1 \leq j \leq l$ and for any $j$ that correspond to loop terms appearing in $a$.  The disk equations for other vertices of $\Gamma'$ are identical to equations for the corresponding vertices of $\Gamma$.  Note that the variables, $x_1, \ldots, x_s$, do not appear in disk equations of $\Gamma'$.  Therefore, a solution of $P(\Gamma)$ is determined by a solution of $P(\Gamma')$ together with an arbitrary choice of $x_2, \ldots, x_s \in \mathbb{F}^\times$ (which uniquely specifies the value of $x_1$ according to $P_i'$).  Thus, changing variables as prescribed above and projecting out the $x_1$ coordinate provides the desired isomorphism.  (That the inverse is also a regular map follows since we have written $x_1$ as a Laurent polynomial in the remaining $x_i$ and $t_i$.) 
  
\item[Case 2:]  Edges $e_1, \ldots, e_s$ have type $N$, and edges $e_{s+1}, \ldots, e_l$ have type $D$ where $1 < s \leq l$.

The procedure is similar.  In the disk equations for $\Gamma$, solve for $x_1$ in $P_i$, and substitute the result into $P_j$.  After the substitution, the variables $x_1,\ldots, x_s$ cancel in $P_j$, and this time the resulting equation is precisely the equation corresponding to the vertex $\tilde{v}$ in the disk equations for $\Gamma'$.  Again we have a correspondence between solutions of $P(\Gamma)$ and solutions of $P(\Gamma')$ together with arbitrary choices of $x_2,\ldots, x_s$, and this completes the proof.
\end{itemize} 
\end{proof}

With this preparation we now give a computation of the solution set $Z_\rho$.

\begin{theorem} \label{thm:diskEq}
Let $\rho$ be an $m$-graded normal ruling of a front diagram $D$.  The affine algebraic set $Z_\rho \subset (\F^\times)^c \times (\F^\times)^n \times \F^{r}$ in bijection with $m$-graded SR-form MCSs of $D$ compatible with $\rho$ satisfies
\[
Z_\rho \cong (\F^\times)^{j(\rho) + c} \times \F^{r}
\]
where $j(\rho) = \#\mbox{switches} - \#\mbox{right cusps}$; $c$ is the number of components of $L$; and $r$ is the number of $m$-graded returns (resp. sum of the number of returns and number of right cusps) when $m \neq 1$ (resp. when $m=1$). 
\end{theorem}

\begin{proof}
Let $\Gamma = (V,E)$ be a connected ruling graph with vertices labeled by invertible elements of $\F[t^{\pm1}_1, \ldots, t^{\pm1}_b]$ with $b\geq 1$, so that the product of labels from all the vertices has the form $\pm t_1^{\pm1} \cdots t_b^{\pm1}$.  We show by induction on $|V|$ that  
\begin{equation} \label{eq:bcase}
Z(\Gamma) \cong (\F^\times)^{|E| - |V| + b}. 
\end{equation}
In the base case, $|V|= 1$, we can solve the single disk equation for one of the $t_i$ that appears with non-zero exponent in the label of $v_1$.  Solutions are then uniquely specified by arbitrary values of $t_j$ and any edge variables $x_j$ (which correspond to loops), so the result follows.  The inductive step follows easily from contracting along an edge and applying Lemma \ref{lem:contract}.

Now, consider the ruling graph, $\Gamma_{\rho} = (V, E)$, that is associated to $\rho$, and write $\Gamma_1, \ldots, \Gamma_N$ for the connected components of $\Gamma_{\rho}$.  When forming the disk equations for $\Gamma_j$ only include those $t_i$ variables that correspond to base points of $L$ located on ruling disks that form the vertices of $\Gamma_j$.  With this caveat, the solution set $Z(\Gamma_{\ruling})$ is the product $Z(\Gamma_1) \times \cdots \times Z(\Gamma_N)$.  In addition, note that the disks that form the vertices of each of the $\Gamma_i$ have as their union some number of components of the link $D$.  Therefore, up to reordering the $t_i$,  the product of labels of the vertices of each $\Gamma_i$ satisfies the hypothesis needed for (\ref{eq:bcase}), and applying (\ref{eq:bcase}) gives
\[
Z(\Gamma_\rho) \cong Z(\Gamma_1) \times \cdots \times Z(\Gamma_N) \cong (\F^\times)^{|E| - |V| + c} \cong (\F^\times)^{j(\rho) + c}.
\]
Finally, the result follows from the observation in equation (\ref{eq:solutions}).
\end{proof}

%% file: Sections/AFormsAndAugmentations.tex
\section{A-form MCSs and Augmentations}
\mylabel{s:AFormAndAugs}

We now introduce a second special class of MCSs known as A-form MCSs.  The name derives from a close relationship between A-form MCSs and augmentations, and in fact a bijective statement is given in Theorem \ref{thm:AugAformbijection}.  We prove this result over a general commutative ring as, for instance, the case $R = \Z$ may be of some independent interest. 

Let $\front$ be a front diagram for a Legendrian link $L$ that is provided with a Maslov potential and fixed orientation.  As usual,  $m$ is a divisor of $2r(L)$, and  we say that a crossing $q$ is $m$-graded if $|q| =0$ mod $m$. 
\begin{definition}
\mylabel{defn:A-form}
Suppose $\MCS$ is an $m$-graded MCS over $R$ for $\front$.  We say that $\MCS$ is in \textbf{A-form} if all left cusps are simple and the only handleslides of $\MCS$ are arranged as follows. 
\begin{enumerate}
	\item Immediately to the left of each $m$-graded crossing there is a handleslide connecting the two crossing strands.  The coefficient may be an arbitrary element of $R$. 
	\item If $m =1$, then handleslides also appear to the left of right cusps with endpoints on the two strands that meet at the cusp.
\end{enumerate}
\end{definition}
Figure \ref{fig:AFormHandle} illustrates the locations of handleslide marks in an A-form MCS.  

\begin{figure}[t]
\labellist
\small\hair 2pt
\pinlabel {$b_1$} [b] at 160 185
\endlabellist
\centering
\includegraphics[scale=.5]{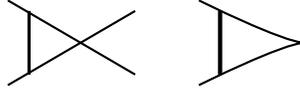}
\caption{The handleslides of an A-form MCS are located to the left of crossings and, when $m=1$, right cusps.}
\label{fig:AFormHandle}
\end{figure}

We let $\mathit{MCS}_m^{A}(\front;R)$ denote the set of all $m$-graded A-form MCSs over $R$ for the front diagram $\front$. Recall that we have denoted by $\overline{Aug}_m(D;R)$ the set of $m$-graded augmentations of the Chekanov-Eliashberg algebra, $(\A(D),\partial)$.

\begin{theorem}
\mylabel{thm:AugAformbijection}
If $D$ is nearly plat, then there exists a bijection $\Theta: \overline{Aug}_m(D;R) \to MCS_m^{A}(D;R)$.
\end{theorem}

\begin{remark}
The nearly plat assumption shortens the proof although we do not expect that it is necessary.  Without it, some additional care should be taken at right cusps.
\end{remark}

\begin{proof}

Let $\varepsilon: \A \rightarrow R$ be a ring homomorphism, and let  $q_1, \ldots, q_N$ denote the $m$-graded crossings of $D$ together with the right cusps in the case $m=1$.  Suppose that we define a handleslide set, $H$, arranged as in the definition of an A-form MCS by taking the coefficient $\lambda_i$ corresponding to the handleslide associated to $q_i$ to be
\begin{equation} \label{eq:lambdai}
\lambda_i = \alpha_i \varepsilon(q_i)
\end{equation}
 where when $q_i$ is a crossing, $\alpha_i$ is $1$ if the understrand of $q_i$ is oriented left and $-1$ if it is oriented to the right; see Figure \ref{fig:ab}. When $q_i$ is a right cusp take $\alpha_i=1$.    Moreover, we assign signs $\ell_1, \ldots, \ell_c$ to marked right cusps as indicated in Figure \ref{fig:rcsign}.

\begin{claim} \label{claim:Aform}  Then, $H$ is the handleslide set of an A-form MCS that assigns values $s_1^{\ell_1}, \ldots, s_c^{\ell_c} \in R^\times$ to marked points $*_1, \ldots, *_c$ if and only if $\varepsilon$ is an augmentation of $(\A(D), \partial)$ with $\varepsilon(t_i) = s_i$, $1 \leq i \leq c$. 
\end{claim}

The claim clearly produces the desired bijection. Its proof requires some preliminary considerations, including Lemma \ref{lem:Cndn} below. 

 Let $\Delta(x_l, i, j ; b_1, \ldots, b_n)$ denote a set of disks up to equivalence  as in the definition of the Chekanov-Eliashberg algebra and subject to the following restrictions.  The image of a disk in $\Delta(x_l, i, j ; b_1, \ldots, b_n)$ is required to lie in the half space of the $xy$-plane where $x \leq x_l$.  Boundary punctures appear in counter-clockwise order along $\partial D^2$ at $z_{-1},z_0, \ldots, z_n$, and the arc of $\partial D^2$ between $z_{-1}$ and $z_0$ maps to the vertical line segement between the $j$ and $i$ strands at $x = x_l$.  The remaining components of $\partial D^2 \setminus \{z_{-1},z_0, \ldots, z_n\}$ map to the Lagrangrian projection of $L$, and neighborhoods of $z_1, \ldots, z_n$ map to negative quadrants at $b_1, \ldots, b_n$ respectively.  See Figure \ref{fig:HalfDisk}.

\begin{figure}[t]
\centering
\[
\begin{array}{rcrc}   
 &\includegraphics[scale=.6]{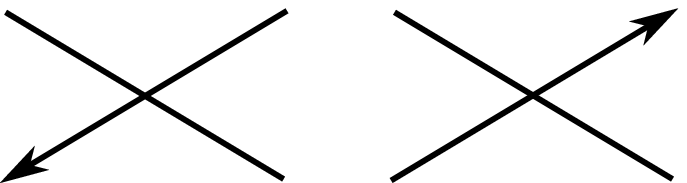} & \quad \quad \quad  & \includegraphics[scale=.6]{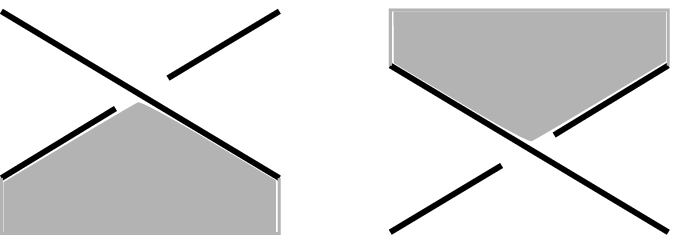} \\
 \alpha_i= & +1 \quad \quad \,\, \quad \quad -1   & \beta_i= & +1 \quad \quad \,\, \quad \quad -1  
\end{array}
\]
\caption{The signs $\alpha_i$ and $\beta_i$ used in the proof of Theorem \ref{thm:AugAformbijection}.   Note that the product of the $\alpha$ sign and $\beta_i$ is the orientation sign $\epsilon_i$ associated to a negative corner of a disk in the definition of the Chekanov-Eliashberg algebra.}
\label{fig:ab}
\end{figure}

Given a disk $f \in \Delta(x_l, i, j ; b_1, \ldots, b_n)$, we let 
\begin{equation} \label{eq:vf}
v(f) = (\beta_1 \lambda_{m_1}) \cdots (\beta_n \lambda_{m_n})
\end{equation}
where $\lambda_{m_l}$ is the coefficient of the handleslide to the left of $b_l$ and $\beta_l$ is $+1$ (resp. $-1$) if a neighborhood of $z_l$ maps to the lower (resp. upper) quadrant of $b_l$; see Figure \ref{fig:ab}.

\begin{figure}[t]
\labellist
\small\hair 2pt
\pinlabel {$b_1$} [b] at 160 185
\pinlabel {$b_2$} [b] at 45 190
\pinlabel {$b_3$} [t] at 116 33
\pinlabel {$x=x_l$} [t] at 236 -4
\pinlabel {$i$} [l] at 240 152
\pinlabel {$j$} [l] at 240 64
\endlabellist
\centering
\includegraphics[scale=.5]{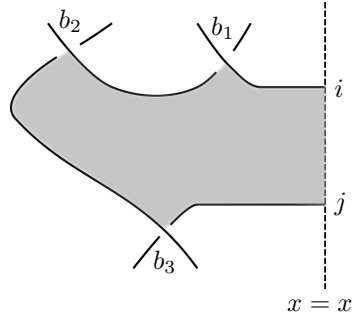}
\caption{A disk in $\Delta(x_l,i,j;b_1, b_2, b_3)$.}
\label{fig:HalfDisk}
\end{figure}

As left cusps of an A-form MCS are simple, Propositions \ref{prop:simple} and \ref{prop:H} apply, and we can attempt to inductively construct the complexes $(C_l,d_l)$ of an A-form MCS associated to the handleslide set $H$.  

\begin{lemma} \label{lem:Cndn}
Suppose the sequence of complexes $(C_k,d_k)$ associated to $H$ can be defined for all $x_k$ with $k \leq l$ and that $x_l$ is immediately to the right of a crossing or left cusp. 
Then, the differential in the complex $(C_l,d_l)$ satisfies
\begin{equation} \label{eq:dnei}
\langle d_l e_i, e_j\rangle = \sum_{ f \in \Delta(x_l, i, j ; b_1, \ldots, b_n)} v(f).
\end{equation}

\end{lemma}
\begin{proof}
The diagram $D$ is nearly plat, and for values of $l$ with only left cusps appearing to the left of $x_l$ the equality follows from the requirement at left cusps in  Definition \ref{defn:MCS}~(4).  Now, assume that $x_l$ is to the right of a crossing, $b_0$, between strands $k$ and $k+1$
and that the Lemma is known for smaller values of $l$.  If the crossing is $m$-graded then it is preceded by a handleslide with coefficient $\lambda$.  We can simultaneously address the case that $b_0$ is not $m$-graded by treating it as if a handleslide with coefficient $0$ precedes $b_0$. Let $d'$ denote the differential prior to this handleslide, and note that (\ref{eq:dnei}) may be assumed to hold for $d'$.  As long as $\langle d' e_k, e_{k+1}\rangle =0$, the complex $(C_l,d_l)$ may be defined, and we can compute
\begin{equation}\label{eq:dneqs}
\begin{array}{lr}  \langle d_l e_i, e_j \rangle = \langle d' e_i, e_j  \rangle & \mbox{ unless $\{i,j\}\cap \{k,k+1\} \neq \emptyset$;} \\
\langle d_l e_k, e_{k+1} \rangle = 0; & \\
\langle d_l e_i, e_{k+1}\rangle = \langle d' e_i, e_{k}\rangle & \mbox{ for $i < k$;} \\
\langle d_l e_i, e_{k}\rangle = \langle d' e_i, e_{k+1} \rangle - \lambda  \langle d' e_i, e_{k}\rangle & \mbox{ for $i<k$;} \\
\langle d_l e_k, e_{j}\rangle = \langle d' e_{k+1}, e_{j}\rangle & \mbox{ for $k+1 < j$;} \\
\langle d_l e_{k+1}, e_{j}\rangle = \langle d' e_{k}, e_{j} \rangle + \lambda  \langle d' e_{k+1}, e_{j}\rangle & \mbox{ for $k+1<j$.} \\
\end{array}
\end{equation}

Notice that, except when $(i,j) = (k,k+1)$, any disk in $\Delta(x', i, j ; b_1, \ldots, b_n)$ may be extended along the Lagrangian projection of $L$ to arrive at a disk in $\Delta(x_l, \sigma(i), \sigma(j) ; b_1, \ldots, b_n)$ where $\sigma = (k \,\, k+1)$ denotes the transposition.  It is easy to see that this defines a bijection 
\[
\Delta(x_l, i, j ; b_1, \ldots, b_n) \cong \Delta(x', \sigma(i),\sigma(j); b_1, \ldots, b_n)
\]
 unless $i = k+1$; $j=k$; or $(i,j) = (k,k+1)$.  The bijection preserves $v(f)$, so in combination with the first, third, and fifth equality of (\ref{eq:dneqs}) this establishes (\ref{eq:dnei}) for such $i$ and $j$.   When $(i,j) = (k,k+1)$, we just note that $\Delta(x_l, k, k+1 ; b_1, \ldots, b_n)$ is empty.  

For the case $j=k$, we have bijections 
\[
\Delta(x_l, i,k; b_1, \ldots, b_n) \cong \Delta(x', i,k+1; b_1, \ldots, b_n) \quad \mbox{and} \quad \Delta(x_l, i,k; b_1, \ldots, b_n, b_0) \cong \Delta(x', i,k; b_1, \ldots, b_n). 
\]
where the first bijection is the one described above and the second bijection arises from extending a disk in $\Delta(x', i,k; b_1, \ldots, b_n)$ to have a negative corner at $b_0$.  The first bijection preserves $v(f)$ and the second multiplies each $v(f)$ by a factor of $-\lambda$.  The negative sign arises from   $\beta_0$ since disks constructed in this manner necessarily cover the upper quadrant of $b_0$.  Combining these observations with the fourth equality of (\ref{eq:dneqs}) establishes (\ref{eq:dnei}).  A similar argument applies in the case $i = k+1$.
\end{proof}

To prove Claim \ref{claim:Aform}, note that $\varepsilon$ is an augmentation if and only if $\varepsilon \circ \partial (q_i)= 0$ for all crossings and right cusps $q_i$.  On the other hand, Proposition \ref{prop:H} tells us that $H$ gives an A-form MCS with values $s_i^{\ell_i}$ assigned to marked points if and only if  immediately to the left of a crossing (resp. right cusp) between the strands $k$ and $k+1$ we have $\langle d_l e_k, e_{k+1} \rangle =0$ (resp.   $\langle d_l e_k, e_{k+1} \rangle$ is equal to $-1$ or $-s_i^{\ell_i}$ depending on if the cusp is unmarked or marked).  We now check that these conditions are equivalent.

Let $q_i$ be a crossing such that the complexes $(C_l,d_l)$ associated to the handleslide set $H$ can be defined to the left of $q_i$.  Take $l$ so that $x_l$ is directly to the left of $q_i$.  We claim that $\langle d_l e_k, e_{k+1} \rangle =0$  if and only if $\varepsilon \circ \partial (q_i) = 0$.   To verify, note that to any disk in $\Delta(x_l, k,k+1; b_1, \ldots, b_n)$ we can associate a disk in $\Delta(q_i; b_1, \ldots, b_n)$ by extending the disk to include a positive corner at the left quadrant of $q_i$.  It is a consequence of the form of the Lagrangian diagram arising from Ng's resolution procedure that this construction gives a bijection.  (There are no disks contributing to $\partial q_i$ with a positive corner in the right quadrant at $q_i$; see \cite{Ng2003}.)  Using this bijection, (\ref{eq:lambdai}) and (\ref{eq:vf}) allow us to compute  
\begin{equation} \label{eq:vfbaq} 
v(f) = (\beta_1 \alpha_{m_1} \varepsilon(q_{m_1}) ) \cdots
 \cdots (\beta_n \alpha_{m_1} \varepsilon(q_{m_1}))
\end{equation}
where $b_i = q_{m_i}$, and $\beta_i$ and $\alpha_{m_i}$ are the signs defined earlier in the proof (see Figure \ref{fig:ab}).  Now, with $w(f)$ and signs $\epsilon',\epsilon_0, \epsilon_1, \ldots, \epsilon_n$ as in (\ref{eq:word}) in Section~\ref{ss:CE-DGA}, we see that (\ref{eq:vfbaq}) becomes
\[
v(f) = \epsilon' \epsilon_0 \varepsilon(w(f))
\]
since we have $\beta_i \alpha_{m_i} = \epsilon_i$ for $1 \leq i \leq n$ and $f$ cannot intersect the marked points $*_i$ because of their placement on right cusps together with the nearly plat assumption.  
Note that the sign $\epsilon' \epsilon_0$ is independent of $f$, so summing over  $f$ and applying Lemma \ref{lem:Cndn} gives that $\langle d_l e_k, e_{k+1} \rangle=0$ if and only if $\varepsilon \circ \partial(q_i)=0$.
  (Strictly speaking, Lemma \ref{lem:Cndn} gives a computation of $\langle d_l e_k, e_{k+1} \rangle$ for $x_l$ to the left of the handleslide preceding $q_i$.  However, passing such a handleslide cannot effect the value of $\langle d_l e_k, e_{k+1} \rangle$.) 

\begin{figure}[t]
\centering
\[
\begin{array}{ccrc}   
\includegraphics[scale=.6]{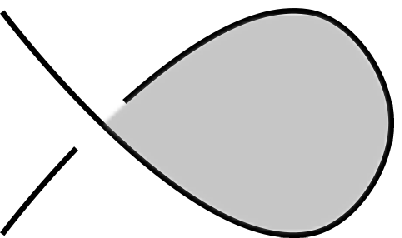} & \quad \quad \quad &  &\includegraphics[scale=.6]{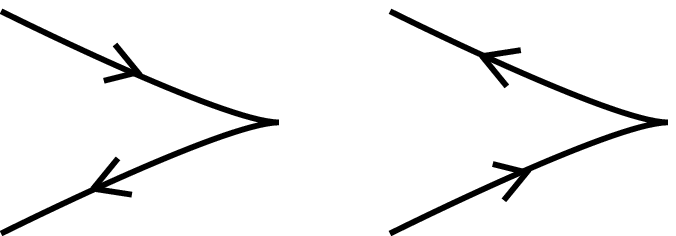}  \\
 f_0 & & \ell_i= & +1 \quad \quad \,\, \quad \quad -1   
\end{array}
\]
\caption{The exceptional disk, $f_0$, (left) and signs at marked right cusps (right).}
\label{fig:rcsign}
\end{figure}

When $q_i$ is a right cusp, we have a bijection $\Delta(q_i; b_1, \ldots, b_n) \cong \Delta(x_l, k,k+1; b_1, \ldots, b_n)$  when $n>0$.  In the case $n=0$,  $\Delta(q_i) \cong \Delta(x_l, k,k+1)   \cup \{f_0\}$  where $f_0$ is the unique disk in $\Delta(q_i)$ that appears to the right of $q_i$; see Figure \ref{fig:rcsign}.  The disk $f_0$ has $w(f_0) = 1$ if the cusp does not contain a marked point and $w(f_0) =t_i^{\ell_i}$ if the cusp contains the marked point $*_i$ where, as in Figure \ref{fig:rcsign}, the exponent $\ell_i$ is $+1$ (resp. $-1$) if the orientation of $D$ is from the upperstrand to the lowerstrand (resp. lowerstrand to the upperstrand) at the cusp.  Note that although Lemma \ref{lem:Cndn} only applies when $x_l$ is to the right of a crossing, the formula given there for $\langle d_l e_k, e_{k+1} \rangle$ remains valid next to the cusp since any intermediate right cusps or handleslides cannot effect this value.  Thus, we can compute
\[
\langle d_l e_k, e_{k+1} \rangle = \sum_{ f \in \Delta(x_l, k,k+1; b_1, \ldots, b_n)} v(f) = (\epsilon'\epsilon_0) \sum_{f \in \Delta(q_i; b_1, \ldots, b_n), f \neq f_0} \varepsilon(w(f)) = \varepsilon\circ \partial(q_i) - \varepsilon(w(f_0)).  
\]
It follows that $\varepsilon\circ \partial(q_i) = 0$ if and only if $\langle d_l e_k, e_{k+1} \rangle = -s_i^{\ell_i}$.
\end{proof}

%% file: Sections/SRFormAFormBijection.tex
\section{A bijection between SR-form and A-form MCSs}
\mylabel{s:MCS2}

In this section we construct an explicit bijection between $MCS_m^{SR}(D;\F)$ and $MCS_m^{A}(D;\F)$.  The paper is then concluded with the proof of Theorem \ref{thm:structure}.  

\begin{theorem}
\mylabel{thm:SRAbijection}
Suppose $\front$ is the nearly plat front diagram of a Legendrian link $\Leg$ with a fixed Maslov potential and $m$ is a divisor of $2 r(\Leg)$. Then there exist well-defined maps $$\Phi: MCS_m^{SR}(\front;\field) \to MCS_m^{A}(\front;\field)$$ and $$\Psi: MCS_m^{A}(\front;\field) \to MCS_m^{SR}(\front;\field)$$ so that $\Psi = \Phi^{-1}$ holds.
\end{theorem}

The proof occupies the next three subsections.  We begin by describing several local modifications that may be applied to handleslide sets of an MCS to produce other MCSs.  In Section~\ref{sec:PhiPsiDefn}, the maps $\Phi$ and $\Psi$ are defined by applying particular sequences of these modifications.  Finally, the proof that $\Psi = \Phi^{-1}$ holds is given in Section~\ref{sec:SRAIdentity}.

\begin{proof}

\begin{figure}[t]
\labellist
\small\hair 2pt
\pinlabel {$r$} [br] at 220 52
\pinlabel {$-r^{-1}$} [br] at 280 52
\pinlabel {$s$} [br] at 143 52
\pinlabel {$s$} [br] at 480 52
\pinlabel {$r$} [br] at 557 52
\endlabellist
\centering
\includegraphics[scale=.6]{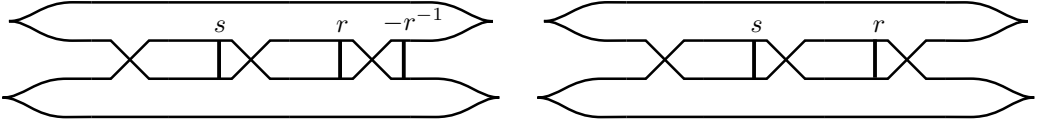}
\caption{Left: An SR-form MCS $\MCS$ with an (R1) graded return at the second crossing and an (S1) switch at the third crossing. Right: The A-form MCS $\Phi(\MCS)$ that is the result of applying the map $\Phi$ defined in Theorem~\ref{thm:SRAbijection} to the left MCS.}
\label{fig:trefoilMCS}
\end{figure}

Recall from Proposition~\ref{prop:simple}, that both A-form and SR-form MCSs are uniquely determined by their handleslide sets. Therefore, we will represent an A-form or SR-form MCS as a front diagram $\front$ with handleslides and suppress the chain complexes; see Figure~\ref{fig:trefoilMCS}. 

\begin{figure}[t]
\labellist
\small\hair 2pt
\pinlabel {$r_1 + r_2$} [br] at 55 120
\pinlabel {$r_1$} [br] at 196 120
\pinlabel {$r_2$} [br] at 275 120
\pinlabel {(a)} [br] at 155 135

\pinlabel {$r_1$} [br] at 336 115
\pinlabel {$r_2$} [bl] at 385 115
\pinlabel {$r_2$} [br] at 500 115
\pinlabel {$r_1$} [bl] at 555 115	
\pinlabel {(b)} [br] at 460 135

\pinlabel {$r$} [br] at 628 115
\pinlabel {$r$} [bl] at 862 115
\pinlabel {(c)} [br] at 762 135

\pinlabel {$r_1$} [br] at 31 26
\pinlabel {$r_2$} [bl] at 78 0
\pinlabel {$r_2$} [br] at 198 0
\pinlabel {$-r_1 r_2$} [bl] at 224 0
\pinlabel {$r_1$} [bl] at 245 26
\pinlabel {(d)} [br] at 155 29

\pinlabel {$r_1$} [br] at 335 0
\pinlabel {$r_2$} [bl] at 383 26
\pinlabel {$r_2$} [br] at 500 26
\pinlabel {$r_1 r_2$} [bl] at 526 26
\pinlabel {$r_1$} [bl] at 547 0
\pinlabel {(e)} [br] at 458 29

\pinlabel {$r$} [br] at 628 26
\pinlabel {$r$} [bl] at 862 26
\pinlabel {(f)} [br] at 762 29

\endlabellist
\centering
\includegraphics[scale=.45]{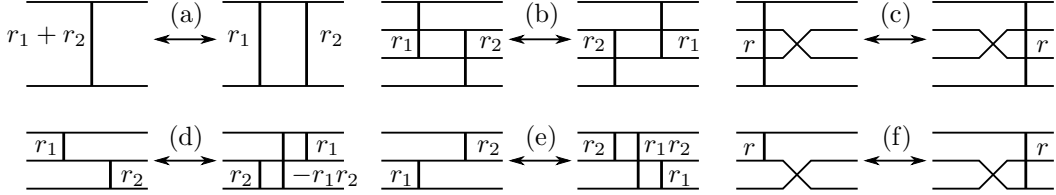}
\caption{Possible local modifications of handleslides in an MCS. The modifications shown do not illustrate all possibilities described in Section~\ref{s:MCS2}.}
\label{fig:MCSequiv}
\end{figure}

We list the local modifications that may be made to the handleslides of an MCS in a tangle $T = \{a \leq x \leq b\} \cap \front$ during the application of $\Phi$ and $\Psi$. After each modification, we are left with an MCS whose chain complexes are equal to those of the MCS before the modification, except possibly for the chain complexes with $x$-coordinates within $T$. 

\begin{enumerate}
	\item[] \textbf{Type 0:} (Introducing or Removing a Trivial Handleslide) In order to simplify the descriptions of the maps $\Phi$ and $\Psi$, we will often find it convenient to introduce or remove a handleslide with coefficient $0$ and endpoints on two strands with the same Maslov potential modulo $m$. We will refer to such a handleslide as a \textbf{trivial} handleslide. Note that the chain isomorphism between consecutive chain complexes in an MCS on either side of a trivial handleslide is the identity.   
	\item[] \textbf{Type 1:} (Sliding a Handleslide Past a Crossing)  Suppose $T$ contains exactly one crossing between strands $k$ and $k+1$ and exactly one handleslide $h$ with endpoints on strands $i$ and $j$, with $i < j$. Recall that we number strands from top to bottom. If $(i, j) \neq (k, k+1)$, then we may slide $h$, either left or right, past the crossing with the endpoints of $h$ remaining on the same strands of $\front$; see Figure~\ref{fig:MCSequiv}~(c)~and~(f) for two such examples. 
	\item[] \textbf{Type 2:} (Interchanging the Positions of Two Handleslides) Suppose $T$ contains exactly two handleslides $h_1$ and $h_2$ with endpoints on strands $i_1 < j_1$ and $i_2 < j_2$ and coefficients $r_1$ and $r_2$, respectively.  We may interchange the positions of the handleslides, so long as $j_1 \neq i_2$ and $i_1 \neq j_2$; see Figure~\ref{fig:MCSequiv}~(b) for one such example. If $j_1 = i_2$ (resp. $i_1 = j_2$) and $h_1$ is left of $h_2$, then we may interchange the positions of $h_1$ and $h_2$, but, in doing so, we must introduce a new handleslide with endpoints $i_1$ and $j_2$ (resp. $j_1$ and $i_2$) and coefficient $-r_1r_2$ (resp. $r_1r_2$); see Figure~\ref{fig:MCSequiv}~(d) for the case $j_1 = i_2$ and Figure~\ref{fig:MCSequiv}~(e) for the case $i_1 = j_2$. 
	\item[] \textbf{Type 3:} (Merging Two Handleslides) Suppose $T$ contains exactly two handleslides $h_1$ and $h_2$ with endpoints on the same two strands and coefficients $r_1$ and $r_2$, respectively. We may merge the two handleslides into one; see Figure~\ref{fig:MCSequiv}~(a). The coefficient of the resulting handleslide is $r_1 + r_2$. 
	\item[] \textbf{Type 4:} (Introducing Two Canceling Handleslides) If $T$ contains no crossings, cusps, or handleslides, then we may introduce two new handleslides with identical endpoints and coefficients $r$ and $-r$, where $r \in \field$.
\end{enumerate}			

We now define an ordering of handleslides within a collection of handleslides. Suppose the tangle $T$ has no crossings or cusps. Number the strands of $T$ from top to bottom $1,2, \hdots, s$. A handleslide $h$ in $T$ has endpoints on two strands of $T$ numbered $t_h$ and $b_h$, where $1 \leq t_h < b_h \leq s$. Given two handleslides $h$ and $h'$, we write $h \prec h'$ if either $t_{h} > t_{h'}$, or $t_h=t_{h'}$ and $b_h < b_{h'}$. A collection of handleslides $V$ in the tangle $T$ is \textbf{properly ordered} if given any two handleslides $h$ and $h'$ in $V$, $h$ is to the left of $h'$ if and only if $h \prec h'$; each dashed box in Figure~\ref{fig:PsiMap} contains a properly ordered collection of handleslides. We let $V^t$ denote the collection that results from reflecting $V$ about a vertical axis whose $x$-coordinate is the midpoint between the left-most and right-most handleslides in $V$. 

We will have need of two additional types of moves. Both are a composition of Type 0, 2 and 3 moves and involve a collection of handleslides $V$, with either $V$ or $V^{t}$ properly ordered, that sits within a tangle $T$ which is assumed to have no crossings or cusps.

\begin{enumerate}
	\item[] \textbf{Type 5:} (Incorporating a Handleslide $h$ into a Collection $V$) 
	Suppose $h$ is a handleslide in $T$ to the \emph{right} of $V$. We define a procedure that moves $h$ into $V$ using Type 2 and 3 moves to create a collection $\overline{V}$ so that if $V$ (resp. $V^t$) is properly ordered, then $\overline{V}$ (resp. $\overline{V}^t$) is properly ordered. Suppose $h$ has coefficient $r$ and $V$ includes a handleslide $h'$ with endpoints on the same strands as $h$ and coefficient $r'$. If there is no such $h'$ in $V$, use a Type 0 move to introduce a handleslide in $V$ with the same endpoints as $h$ and coefficient $r'=0$. The handleslide should be introduced so that $V$ has the same ordering property as before. Label the handleslides between $h$ and $h'$, from \emph{right to left}, by $h_1, h_2, \hdots, h_n$. For each $1 \leq j \leq n$, move $h$ past $h_j$ using a Type 2 move. Doing so may create a new handleslide $h'_j$; see Figure~\ref{fig:MCSequiv} (d) and (e). Merge $h'_j$ with the existing handleslide in $V$ with the same endpoints as $h'_j$ using Type 2 moves and one Type 3 move. The ordering of the handleslides in $V$ ensures that this may be done without introducing any new handleslides. Once $h$ and $h'$ are next to each other, merge them with a Type 3 move. The resulting handleslide has coefficient $r + r'$. An example of this move is given in Figure 25 of \cite{Henry2011} for $\field = \Z / 2\Z$.
	
An analogous procedure may be used to incorporate a handleslide $h$ that is to the \emph{left} of $V$ so that the new collection $\overline{V}$ has the same ordering property as $V$.
	
	\item[] \textbf{Type 6:} (Removing a Handleslide $h$ from a Collection $V$)
	A handleslide $h$ in a collection $V$ may be removed from $V$, using Type 2 moves, so that it appears to either the left or right of the remaining handleslides, denoted $\overline{V}$. As in the case of the Type 5 move, new handleslides may be created by the Type 2 moves; see Figure~\ref{fig:MCSequiv} (d) and (e). However, we may reorder and, if necessary, merge handleslides in $\overline{V}$ using Type 2 and 3 moves so that $\overline{V}$ has the same ordering property as $V$. The ordering property of $V$ ensures that the reordering may be done without introducing any new handleslides. The coefficient of $h$ is unchanged by this process. 

\end{enumerate}

In both the Type 5 and 6 move, the ordering property of $V$ implies that $\overline{V}$ is uniquely determined by $V$ and $h$. This observation is necessary to ensure that both $\Phi$ and $\Psi$ are well-defined.

Finally, recall from Definition~\ref{defn:MCS} that an MCS assigns to each marked point $*_i$ at a right cusp an element $s_i \in \field$. These elements are unchanged by Type 1-6 moves and so we will not have need to consider marked points in the remainder of the proof of Theorem~\ref{thm:SRAbijection}. 

\subsection{Defining $\Phi$ and $\Psi$}
\mylabel{sec:PhiPsiDefn}

We now define the sequence of modifications made by the map $\Phi$ (resp. $\Psi$) to transform an SR-form (resp. A-form) MCS $\MCS$ into an A-form (resp. SR-form) MCS. The sequence of modifications begins to the left of the left-most crossing and ends to the right of the right-most crossing and generalizes the sequence defined in the proofs of Theorem 6.17 and 6.20 in \cite{Henry2011}. The precise construction of $\Phi$ and $\Psi$ is somewhat involved, so we begin with a brief overview. Given an MCS $\MCS$, $\Phi$ and $\Psi$ attempt to slide the handleslides as far right as possible. As we slide the handleslides to the right, we may encounter crossings and other handleslides. We modify the handleslides of $\MCS$ at each such encounter so that we may continue sliding handleslides to the right. An encounter with a crossing may require that we leave behind one or more handleslides before continuing to the right. The modifications made to the MCS in such a case ensure that the handleslides left behind at the crossing adhere to the requirements of the special MCS form we are seeking to create. Thus, as we move from left to right in $\front$, the MCS adheres to one special form to the left of our current position and adheres to the other special form to the right. In the case of the map $\Psi$, we also inductively construct an $m$-graded normal ruling $\ruling$ so that the resulting SR-form MCS $\Psi(\MCS)$ is compatible with $\ruling$. 

Label the crossings of $\front$, from left to right, $q_1, q_2, \hdots, q_n$ and, for each $1 \leq i \leq n$, let $x_i \in \R$ be the $x$-coordinate of $q_i$. Let $x_0$ be the $x$-coordinate of the left-most left cusp and $x_{n+1}$ be the $x$-coordinate of the right-most right cusp. For each $1 \leq i \leq n+1$, let $T_i$ be the Legendrian tangle $\front \cap \{(x,z): x_0 \leq x < x_i\}$. For each $1 \leq i < n$, in defining the map $\Phi$ (resp. $\Psi$) we move a collection of handleslides $V_{\Phi,i}$ (resp. $V_{\Psi,i}$) from the left of $q_i$ to the right and, in so doing, create the handleslide collection $V_{\Phi, i+1}$ (resp. $V_{\Psi, i+1}$). In the case of $\Psi$, we also extend an $m$-graded normal ruling $\ruling$ on $T_i$ to an $m$-graded normal ruling on $T_{i+1}$, also denoted $\ruling$. In the definition of $\Psi$ (resp. $\Phi$), $V_{\Psi, i}$ (resp. $V_{\Phi, i}^t$) is properly ordered for all $1 \leq i \leq n+1$. Since $\Phi$ and $\Psi$ progress from left to right in $\front$, we need only define the process employed by the two maps to create $V_{\Phi, i+1}$ and $V_{\Psi, i+1}$ from $V_{\Phi, i}$ and $V_{\Psi, i}$ and, in the case of $\Psi$, the process employed to extend $\ruling$ from $T_i$ to $T_{i+1}$. The manner in which this is done depends on $V_{\Phi,i}$ and $V_{\Psi,i}$ and whether the crossing $q_i$ is $m$-graded, in the case of $\Phi$, or a switch, $m$-graded return, or $m$-graded departure of the inductively constructed $m$-graded normal ruling $\ruling$, in the case of $\Psi$. The precise description of this process is given in Sections~\ref{ssec:PhiInduction} and \ref{ssec:PsiInduction}. Figures~\ref{fig:PhiMap} and \ref{fig:PsiMap} illustrate one example of this process for $\Phi$ and $\Psi$, respectively. In Section~\ref{ssec:n+1}, we complete the definitions of $\Phi$ and $\Psi$ by  arranging handleslides near the right cusps of $\front$ and, in the case of $\Psi$, proving the inductively constructed $\ruling$ is an $m$-graded normal ruling on \emph{all} of $\front$. 

\begin{figure}[t]
\labellist
\small\hair 2pt
\pinlabel {$V_1 = V_{\Phi, i}$} [br] at 60 165
\pinlabel {$a$} [tl] at 4 134
\pinlabel {$b$} [tl] at 33 118
\pinlabel {$c$} [tl] at 61 102
\pinlabel {$r$} [tl] at 98 118
\pinlabel {$-r^{-1}$} [tl] at 120 120

\pinlabel {(1)} [bl] at 142 130

\pinlabel {$V_2$} [br] at 225 165
\pinlabel {$a$} [tl] at 169 134
\pinlabel {$rc$} [tl] at 193 102
\pinlabel {$b+r$} [tl] at 218 118
\pinlabel {$c$} [tl] at 253 102
\pinlabel {$-r^{-1}$} [tl] at 281 120

\pinlabel {(2)} [bl] at 303 130

\pinlabel {$V_3$} [br] at 390 170
\pinlabel {$a$} [tl] at 365 134
\pinlabel {$-a(b+r)$} [tl] at 320 167
\pinlabel {$b+r$} [tl] at 317 120
\pinlabel {$rc$} [tl] at 386 102
\pinlabel {$c$} [tl] at 413 102
\pinlabel {$-r^{-1}$} [tl] at 442 120

\pinlabel {(3)} [bl] at -2 38

\pinlabel {$V_4$} [tr] at 99 87
\pinlabel {$b+r$} [tl] at 12 26
\pinlabel {$a$} [tl] at 53 26
\pinlabel {$-a(b+r)$} [tl] at 57 44
\pinlabel {$c$} [tl] at 101 10
\pinlabel {$rc$} [tl] at 122 10
\pinlabel {$-r^{-1}$} [tl] at 139 26

\pinlabel {(4)} [bl] at 158 38

\pinlabel {$V_{\Phi,i+1}$} [tr] at 300 87
\pinlabel {$b+r$} [tl] at 173 26
\pinlabel {$a$} [tl] at 226 26
\pinlabel {$-a(b+r)$} [tl] at 229 44
\pinlabel {$0$} [tl] at 273 10
\pinlabel {$-r^{-1}$} [tl] at 290 26

\pinlabel {$rc$} [tl] at 319 10

\endlabellist
\centering
\includegraphics[scale=.95]{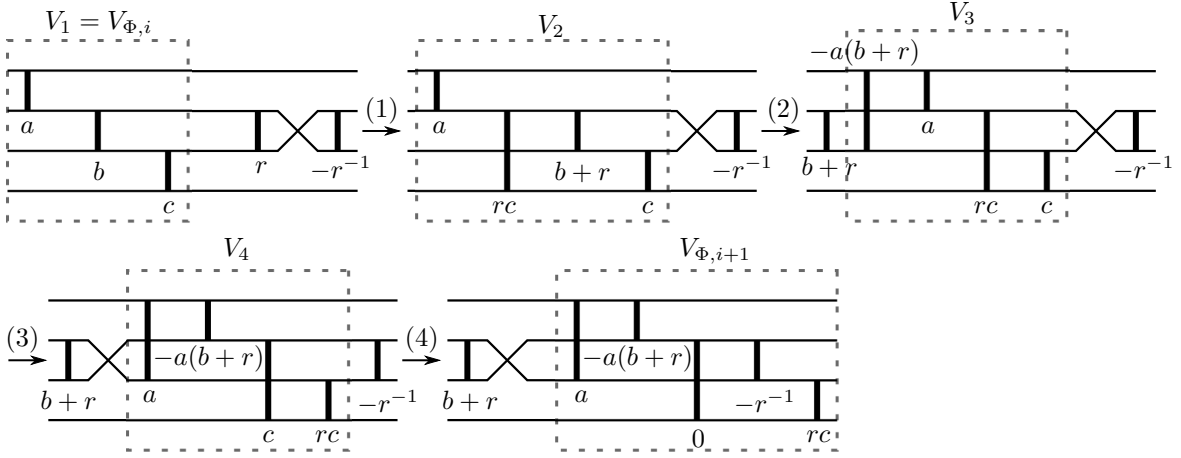}
\caption{The sequence of moves employed by the map $\Phi$ to create $V_{\Phi,i+1}$ from $V_{\Phi,i}$. Coefficients are chosen so that this figure, along with Figure~\ref{fig:PsiMap}, illustrate the proof given in Section~\ref{sec:SRAIdentity} that $\Psi = \Phi^{-1}$.}
\label{fig:PhiMap}
\end{figure}

\subsubsection{Creating $V_{\Phi,i+1}$ from $V_{\Phi, i}$}
\label{ssec:PhiInduction}
Suppose $\MCS$ is an $m$-graded MCS in SR-form compatible with the $m$-graded normal ruling $\ruling_{\MCS}$. Use Type 0 moves to introduce a trivial handleslide just to the left of each $m$-graded departure so that the endpoints of the handleslide are on the strands of the crossing. This is done to simplify the description of $\Phi$. Let $V_{\Phi, 1}$ be the empty collection of handleslides. Suppose that in defining $\Phi$ we have constructed a handleslide collection $V_{\Phi,i}$ and an MCS $\MCS_i$ so that $V_{\Phi,i}$ is to the immediate left of the left-most handleslide near $q_{i}$ and $\MCS_i$ is in SR-form compatible with $\ruling_{\MCS}$ to the right of $V_{\Phi,i}$ and in A-form to the left of $V_{\Phi,i}$. We now define how to create $V_{\Phi, i+1}$ by moving $V_{\Phi,i}$ past $q_i$ so that the resulting MCS $\MCS_{i+1}$ is in SR-form compatible with $\ruling_{\MCS}$ to the right of $V_{\Phi,i+1}$ and in A-form to the left of $V_{\Phi,i+1}$. 

Suppose $|q_i|$ is not congruent to $0$ modulo $m$. Note that $V_{\Phi, i}$ does not include a handleslide with endpoints on the crossing strands $k$ and $k+1$, nor does such a handleslide appear between $q_i$ and $V_{\Phi, i}$. Thus, we may move the handleslides in $V_{\Phi, i}$ past the crossing using a sequence of Type 1 moves. After doing so, the collection of handleslides, now to the right of the crossing, may be reordered using Type 2 moves, without introducing new handleslides, to create the desired collection $V_{\Phi, i+1}$ and MCS $\MCS_{i+1}$.  

Suppose $|q_i|$ is congruent to $0$ modulo $m$. Then there exists a, possibly trivial, handleslide, denoted $h$, with coefficient $r$ between $q_i$ and $V_{\Phi, i}$. The endpoints of $h$ are on the strands $k$ and $k+1$ crossing at $q_i$. Let $V_1$ denote $V_{\Phi, i}$. Use a Type 5 move to incorporate $h$ into $V_1$; see arrow (1) in Figure~\ref{fig:PhiMap}. Let $V_2$ denote the resulting collection of handleslides, $h'$ denote the handleslide in $V_2$ with endpoints on strands $k$ and $k+1$, and $r'$ denote the coefficient of $h'$. Use a Type 6 move to remove $h'$ from $V_2$ so that it sits to the left of the resulting collection $V_3$; see arrow (2) in Figure~\ref{fig:PhiMap}. Move the handleslides of $V_3$ past the crossing using Type 1 moves; see arrow (3) in Figure~\ref{fig:PhiMap}. Reorder the handleslides so that the resulting collection $V_4$ has the same ordering property as $V_1$. The ordering of handleslides in $V_3$ ensures that this may be done without introducing new handleslides. If $q_i$ is a switch or $m$-graded return of $\ruling_{\MCS}$, then there exist up to two handleslides of $\MCS$ to the right of $q_i$ arranged as in Figure~\ref{fig:SRform}. Incorporate these handleslides into $V_4$ using Type 5 moves; see arrow (4) in Figure~\ref{fig:PhiMap}. The resulting collection of handleslides is $V_{\Phi, i+1}$ and the MCS is $\MCS_{i+1}$. The collection $V_{\Phi, i+1}^t$ is properly ordered, since $V_4^t$ was properly ordered. 

Regardless of $|q_i|$, the MCS $\MCS_{i+1}$ is in SR-form compatible with $\ruling_{\MCS}$ to the right of $V_{\Phi,i+1}$ and in A-form to the left of $V_{\Phi,i+1}$. 

\begin{figure}[t]
\labellist
\small\hair 2pt
\pinlabel {$V_1 = V_{\Psi, i}$} [br] at 60 165
\pinlabel {$-c$} [tl] at 0 102
\pinlabel {$-b$} [tl] at 27 118
\pinlabel {$-a$} [tl] at 55 134
\pinlabel {$b+r$} [tl] at 88 120

\pinlabel {(1)} [bl] at 142 131

\pinlabel {$V_2$} [br] at 225 165
\pinlabel {$-c$} [tl] at 165 102
\pinlabel {$r$} [tl] at 198 118
\pinlabel {$-a$} [tl] at 220 134
\pinlabel {$a(b+r)$} [tl] at 238 120

\pinlabel {(2)} [bl] at 303 130

\pinlabel {$V_3$} [br] at 390 165
\pinlabel {$r$} [tl] at 325 118
\pinlabel {$-c$} [tl] at 335 102
\pinlabel {$-rc$} [tl] at 358 102
\pinlabel {$-a$} [tl] at 383 134
\pinlabel {$a(b+r)$} [tl] at 398 120

\pinlabel {(3)} [bl] at -2 38

\pinlabel {$V_4$} [tr] at 99 87
\pinlabel {$r$} [tl] at 21 26
\pinlabel {$-a$} [tl] at 118 26
\pinlabel {$a(b+r)$} [tl] at 85 43
\pinlabel {$-rc$} [tl] at 43 10
\pinlabel {$-c$} [tl] at 72 10

\pinlabel {(4)} [bl] at 139 38

\pinlabel {$V_4$} [tr] at 280 87
\pinlabel {$r$} [tl] at 161 26
\pinlabel {$-r^{-1}$} [tl] at 185 58
\pinlabel {$r^{-1}$} [tl] at 205 28
\pinlabel {$-a$} [tl] at 295 26
\pinlabel {$a(b+r)$} [tl] at 262 43
\pinlabel {$-rc$} [tl] at 220 10
\pinlabel {$-c$} [tl] at 247 10

\pinlabel {(5)} [bl] at 315 38

\pinlabel {$V_{\Psi,i+1}$} [tr] at 460 87
\pinlabel {$r$} [tl] at 338 26
\pinlabel {$-r^{-1}$} [tl] at 356 28
\pinlabel {$-rc$} [tl] at 380 10
\pinlabel {$r^{-1}$} [tl] at 410 26
\pinlabel {$0$} [tl] at 438 10
\pinlabel {$a(b+r)$} [tl] at 447 43
\pinlabel {$-a$} [tl] at 480 26
\endlabellist
\centering
\includegraphics[scale=.95]{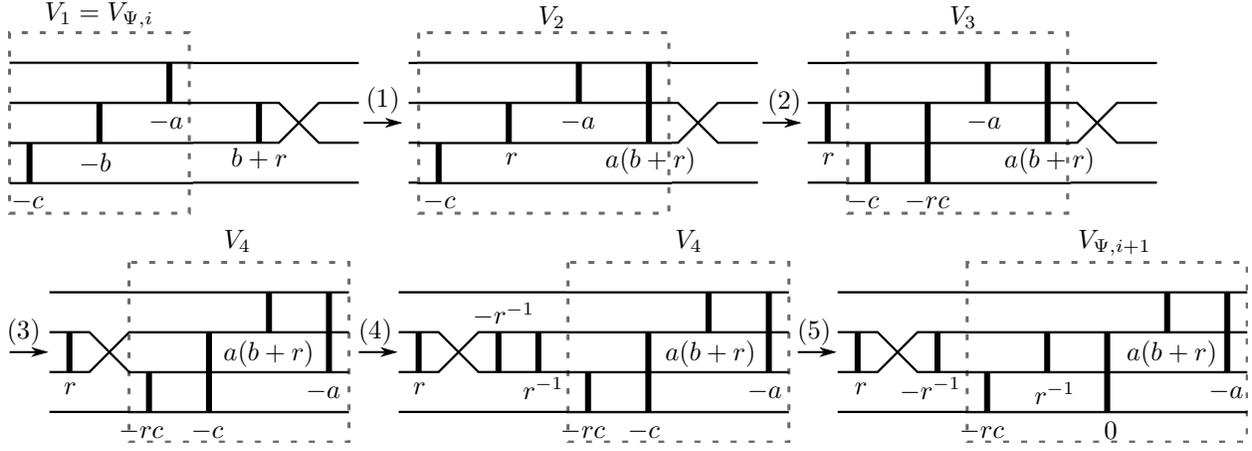}
\caption{The sequence of moves employed by the map $\Psi$ to create $V_{\Psi,i+1}$ from $V_{\Psi,i}$. Coefficients are chosen so that this figure, along with Figure~\ref{fig:PhiMap}, illustrate the proof that $\Psi = \Phi^{-1}$ given in Section~\ref{sec:SRAIdentity}.}
\label{fig:PsiMap}
\end{figure}

\subsubsection{Creating $V_{\Psi,i+1}$ from $V_{\Psi, i}$ and Extending $\ruling$ to $T_{i+1}$}
\label{ssec:PsiInduction}

Suppose $\MCS$ is an $m$-graded MCS in A-form. Let $V_{\Psi, 1}$ be the empty collection of handleslides and define $\ruling$ on $T_1$ to be the $m$-graded normal ruling that pairs strands that terminate together at a left cusp. Suppose that in defining $\Psi$ we have constructed a handleslide collection $V_{\Psi,i}$, an MCS $\MCS_i$, and an $m$-graded normal ruling $\ruling$ on the Legendrian tangle $T_i$ so that $V_{\Psi,i}$ is to the immediate left of the left-most handleslide near $q_{i}$ and $\MCS_i$ is in A-form to the right of $V_{\Psi,i}$ and in SR-form compatible with $\ruling$ to the left of $V_{\Psi,i}$. We now define how to create $V_{\Psi, i+1}$ by moving $V_{\Psi,i}$ past $q_i$ and how to extend the $m$-graded normal ruling $\ruling$ on $T_i$ to an $m$-graded normal ruling on $T_{i+1}$ so that the resulting MCS $\MCS_{i+1}$ is in A-form to the right of $V_{\Psi,i+1}$ and in SR-form compatible with $\ruling$ to the left of $V_{\Psi,i+1}$. 

Suppose $|q_i|$ is not congruent to $0$ modulo $m$.  Then, $V_{\Psi, i+1}$ arises from $V_{\Psi, i}$ in the same manner as in the corresponding case for $\Phi$ discussed above.  In addition, we extend $\ruling$ to an $m$-graded normal ruling on $T_{i+1}$ so that $q_i$ is not a switch.

Suppose $|q_i|$ is congruent to $0$ modulo $m$.  First, handleslide sets $V_1=V_{\Psi, i}$, $V_2$, $V_3$, and $V_4$ are produced precisely as was done for $\Phi$: A Type 5 move pushes the handleslide to the left of $q_i$ into $V_1$ to create $V_2$.  Next, the removal of the handleslide $h'$ with endpoints on the $k$ and $k+1$ strands and coefficient $r'$ produces $V_3$.  Finally, push $V_3$ past the crossing and reorder to arrive at $V_4$.   See arrows (1)-(3) of Figure~\ref{fig:PsiMap}.

Choose $\hat{x} \in \R$ so that $\hat{x}$ is slightly less than the $x$-coordinate of $h'$. At this point in the application of $\Psi$, the MCS $\widehat{\MCS}$, which is the intermediary MCS between $\MCS_i$ and $\MCS_{i+1}$ that includes $V_4$, is in SR-form to the left of $h'$ and compatible with the $m$-graded normal ruling $\ruling$ on the Legendrian tangle $T_i$. Let $\ruling_0$ be the fixed point free involution on the points $\front \cap \{(x,z): x = \hat{x}\}$ defined by $\ruling$. Since $\widehat{\MCS}$ is in SR-form to the left of $h'$, the discussion following Lemma \ref{lem:NormalForm} applies, and the chain complex $(C, d)$ of $\widehat{\MCS}$ with $x$-coordinate $\hat{x}$ is standard with respect to $\ruling$. Thus, $\ruling_0(k) \neq k+1$ holds, since, otherwise, $\langle d e_k, e_{k+1} \rangle \neq 0$ holds, which violates condition (4) in Definition~\ref{defn:MCS}. Let $A = \{k, k+1, \ruling_0(k), \ruling_0(k+1)\}$, $\alpha = \min A$, $\beta = \min (A \setminus \{\alpha, \ruling_0 (\alpha)\})$, $a = \langle d e_{\alpha}, e_{\ruling_0(\alpha)} \rangle$, and $b = \langle d e_{\beta}, e_{\ruling_{0}(\beta)} \rangle$. Since $(C, d)$ is standard with respect to $\ruling$, both $a$ and $b$ are non-zero. We will use $r', a, b \in \field$ and the involution $\ruling_0$ to introduce handleslides to the right of $q_i$ as in Figure~\ref{fig:SRform} and extend the $m$-graded normal ruling $\ruling$ past $q_i$. 

In cases (1)-(4) below, the handleslides and coefficients introduced with Type 4 moves correspond to handleslides described in Definition~\ref{defn:SR-form} (see Figure~\ref{fig:SRform}) so that the MCS $\MCS_{i+1}$ is in A-form to the right of $V_{\Psi,i+1}$ and in SR-form compatible with $\ruling$ to the left of $V_{\Psi,i+1}$, where $\ruling$ has been extended to an $m$-graded normal ruling on $T_{i+1}$.

\begin{enumerate}
	\item Suppose $\ruling_{0}(k) < k < k+1 < \ruling_{0}(k+1)$. If $r'=0$, then define $V_{\Psi, i+1}$ to be $V_4$ and extend $\ruling$ past $q_i$ so that $q_i$ is a departure. If $r' \neq 0$, use a Type 4 move to introduce two handleslides between $q_i$ and the collection $V_4$; see arrow (4) in Figure~\ref{fig:PsiMap}. The handleslides have endpoints on strands $k$ and $k+1$ and the left handleslide has coefficient $-(r')^{-1}$. Incorporate the right handleslide into $V_4$ with a Type 5 move; see arrow (5) in Figure~\ref{fig:PsiMap}. The resulting collection of handleslides is $V_{\Psi, i+1}$ and is properly ordered, since $V_4$ was properly ordered. Extend $\ruling$ past $q_i$ so that $q_i$ is an (S1) switch.
	\item Suppose: 
			\begin{enumerate}
				\item $\ruling_{0}(k+1) < \ruling_{0}(k) < k < k+1$; or
				\item $k < k+1 < \ruling_{0}(k+1) < \ruling_{0}(k)$.
			\end{enumerate}
If $r'=0$, then define $V_{\Psi, i+1}$ to be $V_4$ and extend $\ruling$ past $q_i$ so that $q_i$ is a departure. If $r' \neq 0$, then use a Type 4 move to introduce a pair of handleslides between $q_i$ and the collection $V_4$ with endpoints on $k$ and $k+1$, so that the left handleslide has coefficient $-(r')^{-1}$. Use a second Type 4 move to introduce a pair of handleslides between $q_i$ and $V_4$ with endpoints on the companion strands of $k$ and $k+1$, so that the left handleslide has coefficient $a(r')^{-1}b^{-1}$. Incorporate the right handleslide of each pair into $V_4$ with two Type 5 moves. The resulting collection of handleslides is $V_{\Psi, i+1}$ and is properly ordered. Extend $\ruling$ past $q_i$ so that $q_i$ is an (S2) or (S3) switch in the case of (2a) or (2b), respectively.
	\item Suppose $\ruling_{0}(k+1) < k < k+1 < \ruling_{0}(k)$. Then define $V_{\Psi, i+1}$ to be $V_4$ and extend $\ruling$ past $q_i$ so that $q_i$ is an $m$-graded (R1) return.
	\item Suppose: 
			\begin{enumerate}
				\item $\ruling_{0}(k) < \ruling_{0}(k+1) < k < k+1$; or
				\item $k < k+1 < \ruling_{0}(k) < \ruling_{0}(k+1)$.
			\end{enumerate}
	Use a Type 4 move to introduce two handleslides between $q_i$ and the collection $V_4$. The handleslides have endpoints on the companion strands of strands $k$ and $k+1$ and the left handleslide has coefficient $ar'b^{-1}$, in the case of (a), and $a^{-1}r'b$, in the case of (b). Incorporate the right handleslide into $V_4$ with a Type 5 move. The resulting collection of handleslides is $V_{\Psi, i+1}$ and is properly ordered, since $V_4$ was properly ordered. Extend $\ruling$ past $q_i$ so that $q_i$ is an $m$-graded (R2) or (R3) return, in the case of (4a) or (4b), respectively.
\end{enumerate}	

\subsubsection{Handling $V_{\Phi, n+1}$ and $V_{\Psi, n+1}$}
\label{ssec:n+1}

The collections $V_{\Phi, n+1}$ and $V_{\Psi, n+1}$ appear between the right-most crossing and the left-most right cusp. If $m \neq 1$, then we define that $\Phi$ and $\Psi$ erase all handleslides in $V_{\Phi, n+1}$ and $V_{\Psi, n+1}$, respectively. If $m = 1$, there may exist handleslides in $V_{\Phi, n+1}$ and $V_{\Psi, n+1}$ with endpoints on strands that terminate at the same right cusp. Such handleslides may affect the augmentation associated with an $A$-form MCS and so they must be accounted for. If $m=1$, we define that $\Phi$ and $\Psi$ erase all handleslides in $V_{\Phi, n+1}$ and $V_{\Psi, n+1}$, respectively, except those with endpoints on consecutive strands that terminate at the same right cusp. If $h$ is such a handleslide in either $V_{\Phi, n+1}$ or $V_{\Psi, n+1}$, slide $h$ so that it is just to the left of the right cusp.  If needed, use a Type 3 move to merge $h$ with the handleslide that already appears between the strands at that cusp. 

We now verify that, for both $\Phi$ and $\Psi$, the remaining handleslide set $H$ on $\front$ defines an MCS with simple left cusps. Since $H$ agrees with the handleslides of $\MCS_{n+1}$ near all crossings, Proposition~\ref{prop:H} implies that we need only check that the condition in Definition~\ref{defn:MCS} (4) on the coefficient between strands meeting at a right cusp holds for the chain complexes built inductively from $H$. But, for each right cusp in $\front$, this coefficient is unchanged by the introduction or removal of handleslides between $q_n$ and the right cusps. Thus, $H$ still determines an MCS with simple left cusps. 

The definition of $\Phi$ is now complete and we have shown that $\Phi (\MCS)$ is an A-form MCS. One last step is required in defining $\Psi$. Use Type 0 moves to remove any trivial handleslides that appear just to the left of $m$-graded departures of $\ruling$. The result, $\Psi(\MCS)$, is an MCS, though we must still verify that $\ruling$, the $m$-graded normal ruling built inductively by $\Psi$, is, in fact, an $m$-graded normal ruling on all of $\front$.  For this, it remains to see that strands that meet at right cusps are paired by $\ruling$.  Let $x \in \R$ be slightly less than the $x$-coordinate of the left-most handleslide in $V_{\Psi, n+1}$ and $(C, d)$ be the chain complex of $\Psi(\MCS)$ with $x$-coordinate $x$. By Lemma \ref{lem:NormalForm}, $(C, d)$ is standard with respect to $\ruling$. In order to prove that $\ruling$ is an $m$-graded normal ruling on all of $\front$, we must show that $\langle d e_k, e_{k+1} \rangle \neq 0$ if and only if the $k$ and $k+1$ strands meet at a right cusp. This follows from the requirements in Definition~\ref{defn:MCS} (4) at right cusps of $\Psi(\MCS)$ and the fact that the coefficient $\langle d e_k, e_{k+1} \rangle$ is unchanged by the handleslides between $q_{n}$ and the right cusps. Thus, $\Psi(\MCS)$ is an SR-form MCS compatible with $\ruling$.

\subsection{Proof that $\Psi = \Phi^{-1}$ holds}
\mylabel{sec:SRAIdentity}

\begin{figure}[t]
\labellist
\small\hair 2pt
\pinlabel {$V_{\Psi, i}$} [br] at 85 165
\pinlabel {$-c$} [tl] at 31 102
\pinlabel {$-b$} [tl] at 60 118
\pinlabel {$-a$} [tl] at 87 134

\pinlabel {$V_{\Phi, i}$} [br] at 162 165
\pinlabel {$a$} [tl] at 121 134
\pinlabel {$b$} [tl] at 150 118
\pinlabel {$c$} [tl] at 178 102
\pinlabel {$r$} [tl] at 213 118
\pinlabel {$-r^{-1}$} [tl] at 237 120

\pinlabel {$r$} [tl] at 5 26
\pinlabel {$-r^{-1}$} [tl] at 25 28

\pinlabel {$V_{\Psi,i+1}$} [tr] at 130 87
\pinlabel {$-rc$} [tl] at 46 10
\pinlabel {$0$} [tl] at 105 10 
\pinlabel {$r^{-1}$} [tl] at 78 26
\pinlabel {$a(b+r)$} [tl] at 111 43
\pinlabel {$-a$} [tl] at 147 26

\pinlabel {$V_{\Phi,i+1}$} [tr] at 250 87
\pinlabel {$a$} [tl] at 181 26
\pinlabel {$-a(b+r)$} [tl] at 184 44
\pinlabel {$-r^{-1}$} [tl] at 245 26
\pinlabel {$0$} [tl] at 230 10
\pinlabel {$rc$} [tl] at 274 10

\endlabellist
\centering
\includegraphics[scale=.95]{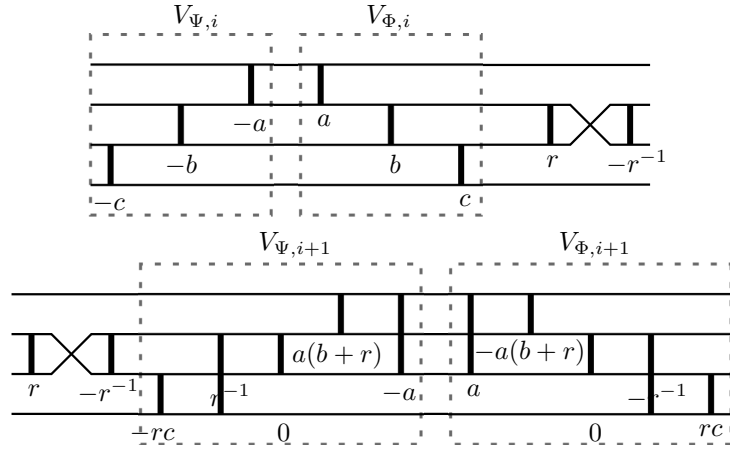}
\caption{If $\Phi$ and $\Psi$ are applied successively to an SR-form MCS $\MCS$, then $V_{\Psi,l} = - V_{\Phi,l}^t$ for all $l$. This can be seen by considering steps (1)-(4) in Figure~\ref{fig:PhiMap}, followed by steps (1)-(5) in \ref{fig:PsiMap}.}
\label{fig:PsiPhi}
\end{figure}

Suppose $\MCS$ is an SR-form $m$-graded MCS. First, we will show $\MCS = \Psi \circ \Phi(\MCS)$. Recall that the crossings in $\front$ are labeled $q_1, \hdots, q_n$, from left to right. Since $\MCS$ and $\Psi \circ \Phi(\MCS)$ are in SR-form, Proposition~\ref{prop:simple} implies that if $\MCS$ and $\Psi \circ \Phi(\MCS)$ have identical handleslides near each crossing and right cusp of $\front$, then the two MCSs are equal. Recall that $\MCS$ and $\Psi \circ \Phi(\MCS)$ are compatible with $m$-graded normal rulings $\ruling_{\MCS}$ and $\ruling_{\Psi \circ \Phi(\MCS)}$. If $|q_i|$ is not congruent to $0$ modulo $m$, then neither $\MCS$, nor $\Psi \circ \Phi(\MCS)$ have handleslides near $q_i$. Suppose $|q_i|$ is congruent to $0$ modulo $m$ and strands $k$ and $k+1$ cross at $q_i$. Let $h_i$ (resp. $h'_i$) be the handleslide of $\MCS$ (resp. $\Psi \circ \Phi(\MCS)$) to the immediate left of $q_i$ with endpoints on the strands $k$ and $k+1$, and  coefficient $r_i$ (resp. $r'_i$); see Figure~\ref{fig:SRform} for the case that $q_i$ is a switch or $m$-graded return of either $\ruling_{\MCS}$ or $\ruling_{\Psi \circ \Phi(\MCS)}$. If $\MCS$ (resp. $\Psi \circ \Phi (\MCS)$) has no such handleslide mark near $q_i$, use a Type 0 move to introduce a trivial handleslide mark in $\MCS$ (resp. $\Psi \circ \Phi (\MCS)$), denoted $h_i$ (resp. $h_i'$), to the left $q_i$ with endpoints on the strands $k$ and $k+1$. In this case, $r_i$ or $r_i'$ is $0$. Recall from Definition~\ref{defn:SR-form} and Figure~\ref{fig:SRform} that the coefficients of handleslides near a switch or $m$-graded return in an SR-form MCS are determined by the coefficient of the left-most handleslide and the chain complexes of the MCS. Thus, Definition~\ref{defn:SR-form} and Proposition~\ref{prop:simple} imply that if $r_i = r'_i$ for every crossing $q_i$ with degree congruent to $0$ modulo $m$, and $\MCS$ and $\Psi \circ \Phi (\MCS)$ have identical handleslides near right cusps, then $\MCS = \Psi \circ \Phi (\MCS)$. Claim~\ref{claim:handleslidesets} below proves $r_i = r'_i$ holds for all such $q_i$.

For each $1 \leq i \leq n$, the application of $\Phi$ to $\MCS$ creates handleslide collections $V_{\Phi,i}(\MCS)$ and the application of $\Psi$ to $\Phi(\MCS)$ creates handleslide collections $V_{\Psi,i}(\Phi(\MCS))$. Both $\Phi$ and $\Psi$ are defined by a sequence of modifications to the handleslides of an MCS and both maps progress from \emph{left to right} in $\front$. Thus, to compare $r_i$ and $r'_i$, we need only consider the effect of first moving $V_{\Phi,i}(\MCS)$ past $q_i$ using $\Phi$ and then moving $V_{\Psi,i}(\Phi(\MCS))$ past $q_i$ using $\Psi$. This is done in the following claim and is illustrated by considering steps (1)-(4) in Figure~\ref{fig:PhiMap} followed by steps (1)-(5) in Figure~\ref{fig:PsiMap}.

\begin{claim}
\label{claim:handleslidesets}
If $1 \leq l \leq n$, then

\begin{enumerate}
	\item $\ruling_{\MCS}$ and $\ruling_{\Psi \circ \Phi(\MCS)}$ agree on the tangle $T_l = \front \cap \{ (x, z):  -\infty < x < x_l \}$, where $x_l$ is the $x$-coordinate of $q_l$;
	\item $V_{\Phi,l}(\MCS) = -V_{\Psi,l}^t(\Phi(\MCS)) $ holds, where $-V_{\Psi,l}^t(\Phi(\MCS))$ is the result of negating each coefficient in $V_{\Psi,l}^t(\Phi(\MCS))$; and 
	\item If $|q_{l}|$ is congruent to $0$ modulo $m$, then $r_{l}$ and $r'_{l}$ are equal, where $r_{l}$ (resp. $r_{l}'$) is the coefficient of the handleslide to the left of $q_l$ in $\MCS$ (resp. $\Psi \circ \Phi (\MCS)$).
\end{enumerate}
 
\end{claim}

\begin{proof}[Proof of Claim \ref{claim:handleslidesets}]
We prove the claim by inducting on $l$. The argument is illustrated in Figures~\ref{fig:PhiMap}, \ref{fig:PsiMap}, and \ref{fig:PsiPhi}. Since both $V_{\Phi,1}(\MCS)$ and $V_{\Psi,1}(\Phi(\MCS))$ are empty, $V_{\Phi,1}(\MCS) = -V_{\Psi,1}^t(\Phi(\MCS)) $ holds and applying $\Phi$ to $\MCS$ and $\Psi$ to $\Phi(\MCS)$ does not affect the handleslide to the left of $q_1$. Therefore, $r_1 = r'_1$ holds. Both $\ruling_{\MCS}$ and $\ruling_{\Psi \circ \Phi(\MCS)}$ must pair strands meeting at a left cusp, so $\ruling_{\MCS}$ and $\ruling_{\Psi \circ \Phi(\MCS)}$ agree on $T_1$. Suppose (1)-(3) hold for all $1 \leq j \leq l $. 


Suppose $|q_l|$ is not congruent to $0$ modulo $m$. Since $\ruling_{\MCS}$ and $\ruling_{\Psi \circ \Phi(\MCS)}$ are $m$-graded, neither can have a switch at $q_i$. Thus, since $\ruling_{\MCS}$ and $\ruling_{\Psi \circ \Phi(\MCS)}$ agree on $T_l$, they must also agree on $T_{l+1}$, and so (1) holds for $l+1$. At $q_l$, $\Phi$ and $\Psi$ use Type 1 moves to move the handleslides of $V_{\Phi,l}(\MCS)$ and $V_{\Psi,l}(\Phi(\MCS))$ past $q_i$ and Type 2 moves to reorder the handleslides after the crossing. There are no new handleslides created during this process. Thus, since (2) holds for $l$, it must also hold for $l+1$. Since (3) is vacuously true when $|q_l|$ is not congruent to $0$ modulo $m$, we have shown (1)-(3) hold for $l+1$ in this case.

Suppose $|q_l|$ is congruent to $0$ modulo $m$. The maps $\Phi$ and $\Psi$ both employ a Type 5 move followed by a Type 6 move to create the collections $V_3(\MCS)$ and $V_3(\Phi(\MCS))$ from $V_{\Phi,l}(\MCS)$ and $V_{\Psi,l}(\Phi(\MCS))$, respectively; see arrows (1) and (2) in Figures~\ref{fig:PhiMap} and \ref{fig:PsiMap}. There is a sign-reversing, endpoint-preserving correspondence between new handleslides created by $\Phi$ and new handleslides created by $\Psi$ during this process. We now detail this correspondence and prove (3) holds for $l+1$. Let $h^0 = h_l$ and $h^4 = h_l'$. Let $h^1$ and $h^3$ be the handleslides of $V_1 (\MCS)$ and $V_1(\Phi(\MCS))$, respectively, with endpoints on strands $k$ and $k+1$. Label the handleslides of $V_{\Phi,l}(\MCS)$ (resp. $V_{\Psi,l}(\Phi(\MCS))$), from right to left (resp. left to right), $g_1, \hdots, g_t$ (resp. $g'_1, \hdots, g'_t$). Since $V_{\Phi,l}(\MCS) = -V_{\Psi,l}^t(\Phi(\MCS))$, for all $1 \leq j \leq t$, $g_j$ and $g'_j$ have endpoints on the same strands and their coefficients, denoted $s_j$ and $s'_j$ respectively, satisfy $s_j = -s'_j$. Suppose $g_{a}$ and $g'_a$ are the handleslides of $V_{\Phi,l}(\MCS)$ and $V_{\Psi,l}(\Phi(\MCS))$ with endpoints on strands $k$ and $k+1$. Then, the handleslide $h^2$ to the right of $V_{\Psi,l}(\Phi(\MCS))$ that results from moving $h^0$ into $V_{\Phi,l}(\MCS)$ with a Type 5 move and $h^1$ out of $V_{1}(\MCS)$ with a Type 6 move has coefficient $r_l + s_a$. The map $\Psi$ moves $h^2$ into $V_{\Psi,l}(\Phi(\MCS))$ with a Type 5 move and out of $V_{1}(\Phi(\MCS))$ with a Type 6 move. The resulting handleslide is $h^4$ and its coefficient is $r'_l = r_l + s_a + s'_a$. Since $s_a = -s'_a$ holds, we see that $r_l' = r_l$ holds and so we have verified that (3) holds for $l+1$. Finally, we note that the identity $V_{\Phi,l}(\MCS) = -V_{\Psi,l}^t(\Phi(\MCS))$, along with the sign difference in Figure~\ref{fig:MCSequiv} (d) and Figure~\ref{fig:MCSequiv} (e) imply the following two biconditional statements. For $1 \leq i < a$, a handleslide with coefficient $\alpha$ and endpoints on strands $\beta$ and $\gamma$ is created when $h_0$ is moved past $g_i$ during the Type 5 move of $\Phi$ if and only if a handleslide with coefficient $-\alpha$ and endpoints on strands $\beta$ and $\gamma$ is created when $h_3$ is moved past $g'_i$ during the Type 6 move of $\Psi$. Likewise, for $a+1 \leq i \leq t$, a handleslide with coefficient $\alpha$ and endpoints on strands $\beta$ and $\gamma$ is created when $h_1$ is moved past $g_i$ during the Type 6 move of $\Phi$ if and only if a handleslide with coefficient $-\alpha$ and endpoints on strands $\beta$ and $\gamma$ is created when $h_2$ is moved past $g'_i$ during the Type 5 move of $\Psi$. Thus, $V_3(\MCS) =-V_3^t(\Phi(\MCS))$ holds. Now, $\Phi$ and $\Psi$ use Type 1 moves to move the handleslides of $V_{3}(\MCS)$ and $V_{3}(\Phi(\MCS))$ past $q_i$ and Type 2 moves to reorder the handleslides after the crossing. There are no new handleslides created during this process. Thus, since $V_{3}(\MCS) = -V_{3}^t(\Phi(\MCS))$ holds, $V_{4}(\MCS) = -V_{4}^t(\Phi(\MCS))$ holds as well. 

It remains to verify (1) and (2) hold for $l+1$. Let $(C, d)$ (resp. $(C', d')$) be the chain complex of $\MCS$ (resp. $\Psi \circ \Phi (\MCS)$) that appears to the left of $h_l$ (resp. $h_l'$). Definition~\ref{defn:SR-form} and Lemma~\ref{lem:NormalForm} imply that the placement and coefficients of handleslides of $\MCS$ (resp. $\Psi \circ \Phi (\MCS)$) that appear near $q_i$ are determined by $r$ (resp. $r'$) and $(C, d)$ (resp. $(C', d')$); see Figure~\ref{fig:SRform}. Note that $(C', d')$ is equal to the chain complex that sits to the immediate left of $V_{\Psi, l} (\Phi(\MCS))$ and $(C, d)$ is equal to the chain complex that sits to the immediate right of $V_{\Phi, l} (\MCS)$. Since $V_{\Phi,l}(\MCS) = -V_{\Psi,l}^t(\Phi(\MCS))$ holds, the handleslides cancel in pairs when $V_{\Phi,l}(\MCS)$ and $V_{\Psi,l}(\Phi(\MCS))$ are placed next to each other. Thus, $(C, d)$ and $(C', d')$ are equal. Since we have already verified $r$ and $r'$ are equal, we may conclude that the placement and coefficients of handleslides of $\MCS$ near $q_l$ is the same as the placement and coefficients of handleslides of $\Psi \circ \Phi (\MCS)$ near $q_l$. This has two consequences. First, since $(C, d) = (C',d')$ holds and $\MCS$ and $\Psi \circ \Phi(\MCS)$ have the same handleslides near $q_l$, the chain complexes of $\MCS$ and $\Psi \circ \Phi(\MCS)$ that sit just to the right of these handleslides are equal as well and both are standard with respect to $\ruling_{\MCS}$ and $\ruling_{\Psi \circ \Phi(\MCS)}$, respectively. Thus, (1) holds for $l+1$; $\ruling_{\MCS}$ and $\ruling_{\Psi \circ \Phi (\MCS)}$ agree on $T_{l+1}$. Secondly, a handleslide incorporated into $V_4(\MCS)$ by $\Phi$, must be replaced with an identical handleslide by $\Psi$. Suppose a handleslide $g$ with coefficient $s$ is to the right of $V_4(\MCS)$ and is incorporated into $V_4(\MCS)$ with a Type 5 move during the application of $\Phi$. Then during the application of $\Psi$, two handleslides are created to the left of $V_4(\Phi(\MCS))$ with coefficients $s$ and $-s$ and the handleslide with coefficient $-s$, denoted $g'$, is incorporated into $V_4(\Phi(\MCS))$ with a Type 5 move. Label the handleslides that $g$ (resp. $g'$) moves past during the Type 5 move employed by $\Phi$ (resp $\Psi$), from right to left (resp. left to right), $g_1, \hdots, g_t$ (resp. $g_1', \hdots, g_t'$). Since $V_4(\MCS)=-V_4^t(\Phi(\MCS))$, for all $1 \leq j \leq t$, $g_j$ and $g_j'$ have endpoints on the same strands and the coefficients of $g_j$ and $g'_j$, denoted $s_j$ and $s'_j$ respectively, satisfy $s_j = -s'_j$. In addition, $g$ (resp. $g'$) moves from right to left (resp. left to right) past $g_j$ (resp. $g_j'$). Thus, since the coefficients of $g$ and $g'$ differ by a sign, $s_j = -s'_j$ holds, and the handleslides created in Figure~\ref{fig:MCSequiv} (d) and Figure~\ref{fig:MCSequiv} (e) have opposite signs, a handleslide with coefficient $\alpha$ is created by moving $g$ past $g_j$ if and only if a handleslide with coefficient $-\alpha$ is created by moving $g'$ past $g_j'$. Thus, we may conclude (2) holds for $l+1$; $V_{\Phi,l+1}(\MCS)$ and $-V_{\Psi,l+1}^t(\Phi(\MCS))$ are equal.
\end{proof}

We have shown that, for any switch or $m$-graded return $q_i$, $r_i = r'_i$ holds. Finally, we must show that $\MCS$ and $\Psi \circ \Phi(\MCS)$ have the same handleslide marks near right cusps. This follows immediately from the definition of $\Phi$ and $\Psi$ and the fact that $V_{\Psi,n+1}(\Phi(\MCS)) = -V_{\Phi,n+1}^t(\MCS) $. Thus, $\Psi \circ \Phi = Id$ holds. 

Since $\Psi$ is a left inverse of $\Phi$, $\Psi$ is surjective. If $\Psi$ is injective, then it follows that $\Psi = \Phi^{-1}$. In order to show $\Psi$ is injective, let $\MCS$ and $\MCS'$ be distinct A-form MCSs in $MCS_m^A(\front; \field)$. Find the smallest value $1 \leq i \leq n$ so that the handleslides $h$ and $h'$ of $\MCS$ and $\MCS'$, respectively, to the left of crossing $q_i$ have coefficients $r$ and $r'$, respectively, and $r \neq r'$. Suppose strands $k$ and $k+1$ cross at $q_i$. The MCSs $\MCS$ and $\MCS'$ have identical handleslides near all crossings to the left of $q_i$. Thus, when $\Psi$ is applied to $\MCS$ or $\MCS'$, the same collection of handleslides (denoted $V_{\Psi,i}$ in Figure~\ref{fig:PsiMap}) arrives at $q_i$ from the left. Call this collection $V$ and suppose the handleslide in $V$ between strands $k$ and $k+1$ has coefficient $a$, where $a$ is possibly $0$. The coefficient of the handleslide to the left of $q_i$ in $\Psi(\MCS)$ and $\Psi(\MCS')$ is $a + r$ and $a + r'$, respectively. Since $r \neq r'$, we may conclude $\Psi(\MCS) \neq \Psi(\MCS')$ and $\Psi$ is injective, as desired. 
\end{proof}

In Section~\ref{sec:regularityproof}, we complete the proof of the main result, Theorem~\ref{thm:Main}, by giving the proof of Theorem~\ref{thm:structure}. The proof requires the following proposition. In short, the coefficients of handleslides in the A-form MCS $\Phi(\MCS)$ are given by Laurent polynomials in the coefficients of handleslides in the SR-form MCS $\MCS$, and the Laurent polynomials depend only on the $m$-graded normal ruling compatible with $\MCS$. 

\begin{proposition}
\label{prop:regularity}
Let $\front$ be the nearly plat front diagram of a Legendrian link $\Leg$ with fixed Maslov potential. Let $m$ be a divisor of $2 r(\Leg)$ and $\ruling$ an $m$-graded normal ruling of $\front$. Let $n(\ruling)$ be the number of switches of $\ruling$ and $r(\ruling)$ be the number of $m$-graded returns of $\ruling$ when $m \neq 1$ and the number of $m$-graded returns and right cusps when $m=1$. Label the crossings and right cusps of $\front$, from left to right, by $q_1, \hdots, q_N$. 

Then, for each $1 \leq i \leq N$, there exists a Laurent polynomial $g_i^{\ruling} \in \field[x_1^{\pm1}, \hdots, x_{n(\ruling)}^{\pm1},z_1, \hdots,z_{r(\ruling)}]$ so that, for any SR-form MCS $\MCS \in MCS_m^{\ruling}(\front; \field)$, the coefficient of the handleslide at $q_i$ in the A-form MCS $\Phi(\MCS)$ is $g_i^{\ruling}(r_1, \hdots, r_{n(\ruling)}, u_{1}, \hdots, u_{r(\ruling)})$ where

\begin{enumerate}
	\item For $1 \leq j \leq n(\ruling)$, $r_j \in \field^{\times}$ is the coefficient of the handleslide to the left of the $j^{th}$ switch crossing in $\MCS$; see Figure~\ref{fig:SRform};
	\item For $1 \leq j \leq r(\ruling)$, $u_j \in \field $ is the coefficient of the handleslide to the left of the $j^{th}$ $m$-graded return crossing or right cusp in $\MCS$; see Figure~\ref{fig:SRform} for the case of an $m$-graded return;
\end{enumerate}

\end{proposition}

\begin{proof}
Let $n$ be the number of crossings in $\front$. For each $1 \leq l \leq n+1$, let $V_{\Phi,l}$ be the handeslide collection $V_{\Phi,l}(\MCS)$ as in the construction of the map $\Phi$. Given $1 \leq l \leq n$, we number the strands of $\front$, from top to bottom, $1, 2, \hdots, s_l$ to the left of the $l^{th}$ crossing. Given $1 \leq i < j \leq s_l$, let $b^{l}_{i,j}$ be the coefficient of the handleslide in $V_{\Phi,l}$ with endpoints on strands $i$ and $j$. If no such handleslide exists in $V_{\Phi, i}$, let $b^l_{i,j}$ be $0$. Suppose $q_l$ is a crossing or right cusp and strands $k$ and $k+1$ cross or terminate at $q_i$. Let $r \in \field$ be the coefficient of the handleslide with endpoints on $k$ and $k+1$ just to the left of $q_l$ in $\MCS$. Again, if no such handleslide exists, then let $r=0$. Then, the coefficient of the handleslide at the crossing (resp. right cusp) $q_l$ in $\Phi(\MCS)$ is $b^l_{k,k+1} + r$ (resp. $b^{n+1}_{k,k+1}+r$); an example of a crossing is illustrated in Figure~\ref{fig:PhiMap}. Thus, we need only show that $b^l_{k,k+1}$ is given by a Laurent polynomial in $\field[x_1^{\pm1}, \hdots, x_{n(\ruling)}^{\pm1},z_1, \hdots,z_{r(\ruling)}]$. This is the content of Claim~\ref{claim:regularity}. 

\begin{claim}
\label{claim:regularity}
For all $1 \leq l \leq n$ and $1 \leq i < j \leq s_l$, there exists a Laurent polynomial $h^{l}_{i,j} \in \field[x_1^{\pm1}, \hdots, x_{n(\ruling)}^{\pm1},z_1, \hdots,z_{r(\ruling)}]$ so that, for any SR-form MCS $\MCS \in MCS_m^{\ruling}(\front;\field)$, $$b^l_{i,j} = h^l_{i,j}(r_1, \hdots, r_{n(\ruling)}, u_{1}, \hdots, u_{r(\ruling)})$$ where  $r_1, \hdots, r_{n(\ruling)}, u_{1}, \hdots, u_{r(\ruling)}$ are as defined in the statement of Proposition~\ref{prop:regularity}.
\end{claim}

In order to prove Claim~\ref{claim:regularity}, we will induct on $l$. For $l=1$ and any $\MCS \in MCS_m^{\ruling}(\front; \field)$, $V_{\Phi, 1}$ is empty, so $b^l_{i,j}=0$ for all $1 \leq i < j \leq s_1$. Suppose the claim holds for some fixed $l$. We must show the claim holds for $l+1$ as well. Let $\MCS \in MCS_m^{\ruling}(\front; \field)$. The coefficients of handleslides in $V_{\Phi, l+1}$ are the result of arithmetic operations, specifically addition, subtraction, and multiplication, involving the coefficients of handleslides in $V_{\Phi, l}$ and the coefficients of handleslides near the crossing $q_l$; see Figure~\ref{fig:PhiMap} for the case of an (S1) switch at $q_l$. The arithmetic operations are defined by the Type 1-6 moves; see, for example, Figure~\ref{fig:MCSequiv}. The sequence of moves used by $\Phi$ is determined by the endpoints of handleslides near $q_l$, which depends only on $\ruling$. Let $(C,d)$ be the chain complex of $\MCS$ with $x$-coordinate to the immediate left of the left-most handleslide near $q_l$. If there are no handleslides near $q_l$, let $(C,d)$ be the chain complex of $\MCS$ with $x$-coordinate to the immediate left of $q_l$. By Lemma~\ref{lem:NormalForm}, $(C, d)$ is standard with respect to $\ruling$. The endpoints of handleslides near $q_l$ depend only on $\ruling$ and the coefficient of any such handleslide is a product of coefficients from the differential $d$ in $(C, d)$ and an element of $\{ \pm r_1^{\pm1}, \hdots, \pm r_{n(\ruling)}^{\pm1}, u_{1}, \hdots, u_{r(\ruling)}\}$; see Figure~\ref{fig:SRform}. However, each coefficient in the differential $d$ is determined by evaluating a Laurent polynomial in $\field[x_1^{\pm1}, \hdots, x_{n(\ruling)}^{\pm1}]$ on $(r_1, \hdots, r_{n(\ruling)})$ and this Laurent polynomial depends only on $\ruling$.  Specifically, as in the discussion prior to Theorem \ref{thm:zrho}, the coefficient of the differential that connects the two boundary strands of a disk $D_i$ of $\rho$ is given by a product with one term of the form $\pm r_i^{\pm1}$ for each switch along $D_i$ that occurs prior to $(C,d)$.  Thus, we may conclude that Claim~\ref{claim:regularity} holds for $l+1$.
 
\end{proof}

\subsection{Proof of Theorem~\ref{thm:structure}}
\label{sec:regularityproof}

After the work done in Sections~\ref{s:MCS1}, \ref{s:AFormAndAugs}, and \ref{s:MCS2} we are finally in a position to prove Theorem~\ref{thm:structure}.

\begin{proof}[Proof of Theorem~\ref{thm:structure}]
Let $\ruling$ be an $m$-graded normal ruling of $\front$. By Theorem~\ref{thm:diskEq}, $(\F^{\times})^{j(\rho) + c(L)} \times \F^{r(\rho)}$ and $Z_{\ruling} \subset (\field^{\times})^{c(\Leg)} \times (\field^{\times})^{n(\ruling)} \times \field^{r(\ruling)}$ are isomorphic, as algebraic sets, where $n(\ruling)$ is the number of switches of $\ruling$. Therefore, it suffices to prove that there exists an injective regular map $\varphi_{\ruling} : Z_{\ruling} \hookrightarrow V_m(\front; \field)$ and that $V_m(\front; \field)$ is the disjoint union $\coprod_{\ruling} W_{\ruling}$, where $W_{\ruling}$ is the image of $\varphi_{\ruling}$. We define
\begin{equation*}
\varphi_{\ruling} : Z_{\ruling} \to V_m(\front; \field), \mbox{ by } \alpha \mapsto \Omega \circ \Theta^{-1} \circ \Phi \circ \Lambda (\alpha)
\end{equation*}
where

\begin{enumerate}
\item  $\Lambda : Z_{\ruling} \to MCS_m^{\ruling}(\front; \field)$ is the bijection defined in Theorem~\ref{thm:zrho};
\item $\Phi: MCS_m^{SR}(\front; \field) \to MCS_m^A(\front; \field)$ is the bijection defined in Theorem~\ref{thm:SRAbijection}; 
\item $\Theta^{-1}: MCS^A_m(\front; \field) \to \overline{Aug}_m(\front; \field)$ is the bijection defined in Theorem~\ref{thm:AugAformbijection}; and
\item $\Omega : \overline{Aug}_m(\front; \field) \to V_m(\front; \field)$ is the bijection defined in the discussion after Definition~\ref{defn:aug}.
\end{enumerate}

All four maps in the definition of $\varphi_{\ruling}$ are bijections between their respective domains and codomains. Thus, $\varphi_{\ruling}$ is an injection. In addition, $MCS_m^{SR}(\front; \field)$ is the disjoint union $\coprod_{\ruling} MCS_m^{\ruling}(\front; \field)$, where the union is over all $m$-graded normal rulings of $\front$. Therefore, $V_m(\front; \field)$ is, in fact, the disjoint union $\coprod_{\ruling} W_{\ruling}$, where $W_{\ruling}$ is the image of $\varphi_{\ruling}$. It remains to show that $\varphi_{\ruling}$ is a regular map.

Recall that $Z_{\ruling}$ is an algebraic set in $(\field^{\times})^{c(\Leg)} \times (\field^{\times})^{n(\ruling)} \times \field^{r(\ruling)}$ and $V_m(\front; \field)$ is an algebraic set in $(\field^{\times})^{c(\Leg)} \times \field^N$, where $N$ is the number of crossings and right cusps in $\front$. We must show that there exist Laurent polynomials $$f_1^{\ruling}, \hdots, f_{c(\Leg)+N}^{\ruling} \in \field [t_1^{\pm1}, \hdots, t_{c(\Leg)}^{\pm1},x_1^{\pm1}, \hdots, x_{n(\ruling)}^{\pm1}, z_1, \hdots, z_{r(\ruling)}]$$ so that for any  $\alpha = (s_1, \hdots, s_{c(\Leg)}, r_1, \hdots, r_{n(\ruling)}, u_1, \hdots, u_{r(\ruling)})$ in $Z_{\ruling} \subset (\field^{\times})^{c(\Leg)} \times (\field^{\times})^{n(\ruling)} \times \field^{r(\ruling)}$ 
\begin{equation*}
\varphi_{\ruling} (\alpha) = (f_1^{\ruling}(\alpha), \hdots, f_{c(\Leg)+N}^{\ruling}(\alpha)).
\end{equation*}

Recall from Theorem~\ref{thm:zrho} that the coefficients of handleslides in the SR-form MCS $\Lambda(\alpha)$ are determined by the entries in $\alpha$, up to a sign that depends only on $\ruling$. In particular, the first handleslide coefficient at the $j^{th}$ switch is $r_j$, except when the switch has Type (S3) in which case it is $-r_j$, the first handleslide coefficient at the $j^{th}$ $m$-graded return, or right cusp if $m=1$, is $u_j$, and the marked point on the $j^{th}$ component of $\Leg$ is assigned $s_j$. For  $\Phi \circ \Lambda(\alpha) \in  MCS^A_m(\front; \field)$, $\Theta^{-1}\circ \Phi \circ \Lambda(\alpha)$ is an augmentation in $\overline{Aug}_m(\front; \field)$, which we denote $\aug$. The augmentation $\aug$ evaluated on a crossing or right cusp generator of the Chekanov-Eliashberg algebra is the coefficient of the corresponding handleslide in $\Phi \circ \Lambda(\alpha$), up to a sign that depends only on the front diagram $\front$. The augmentation $\aug$ evaluated on a marked point generator $t_j$ of the Chekanov-Eliashberg algebra is either the element of $\field$ assigned by $\Phi \circ \Lambda(\MCS)$ to the marked point on the $j^{th}$ component of $\Leg$ or the inverse of this element, depending on the orientation of the cusp. The map $\Omega$ sends $\aug$ to $(\aug(t_1), \hdots, \aug(t_{c(\Leg)}), \aug(q_1), \hdots, \aug(q_{N}))$, where $q_1, \hdots, q_N$ are the generators of the Chekanov-Eliashberg algebra corresponding to the crossings and right cusps of $\front$ and $t_1, \hdots, t_{c(\Leg)}$ are the generators corresponding to the marked points. From these observations about $\Lambda, \Theta^{-1},$ and $\Omega$, we may conclude that $\varphi_{\ruling}$ is a regular map if the coefficients of handleslides in the A-form MCS $\Phi\circ \Lambda(\alpha)$ are given by Laurent polynomials in the coefficients of handleslides in the SR-form MCS $\Lambda(\alpha)$, and the Laurent polynomials depend only on the $m$-graded normal ruling $\ruling$. This is exactly the conclusion of Proposition~\ref{prop:regularity}.
\end{proof}